\documentclass[12pt,reqno]{amsart}
\usepackage[cp1251]{inputenc}
\usepackage[notref,notcite]{}
\usepackage{amssymb}
\usepackage{amsmath}

\usepackage{amsfonts,xr,graphicx}

\makeatletter

\@addtoreset{equation}{section}
\makeatother

\newtheorem{theorem}{Theorem}[section]
\newtheorem{lemma}[theorem]{Lemma}
\newtheorem{proposition}[theorem]{Proposition}
\newtheorem{corollary}[theorem]{Corollary}

\def\D{\mathbb D}
\def\C{\mathbb C}
\def\R{\mathbb R}
\def\S{\mathbb S}
\def\N{\mathbb N}

\def\b{\backslash}

\def\l{\langle}
\def\r{\rangle}
\def\Z{\mathbb Z}
\def\T{{\rm Tr}}

\begin{document}

\title[Steklov zeta-invariants and a compactness theorem]{Steklov zeta-invariants and a compactness theorem for isospectral families of planar domains}
\author{Alexandre Jollivet and Vladimir Sharafutdinov}
\thanks{
The first author is partially supported by French grant ANR-13-JS01-0006.\\
This work was partially done when the first author visited Sobolev Institute of Mathematics and Institute of Computational Mathematics and Mathematical Geophysics of the Siberian Branch of RAS in August 2016. The author is grateful to both institutes for their hospitality. This visit was financially supported by
French grant "BQR 2016 Coop\'eration Internationale" from University of Lille 1.\\
The second author was supported by RFBR, Grant 15-01-05929-a.
}
\address{Laboratoire de Math\'ematiques Paul Painlev\'e,
CNRS UMR 8524/Universit\'e Lille 1 Sciences et Technologies,
59655 Villeneuve d'Ascq Cedex, France}
\email{alexandre.jollivet@math.univ-lille1.fr}

\address{Sobolev Institute of Mathematics, Novosibirsk, Russia}
\address{Novosibirsk State University, Russia}
\email{sharaf@math.nsc.ru}



\begin{abstract}
The inverse problem of recovering a smooth simply connected multisheet planar domain from its Steklov spectrum is equivalent to the problem of determination, up to a gauge transform, of a smooth positive function $a$ on the unit circle from the spectrum of the operator $a\Lambda$, where $\Lambda$ is the Dirichlet-to-Neumann operator of the unit disk.
Zeta-invariants are defined by $Z_m(a)=\T[(a\Lambda)^{2m}-(aD)^{2m}]$ for every smooth function $a$. In the case of a positive $a$, zeta-invariants are determined by the Steklov spectrum. We obtain some estimate from below for $Z_m(a)$ in the case of a real function $a$. On using the estimate, we prove the compactness of a
Steklov isospectral family of planar domains in the $C^\infty$-topology. We also describe all real functions $a$ satisfying $Z_m(a)=0$.
\end{abstract}

\maketitle

\section{Introduction}

Let $\Omega$ be a simply connected (possibly multisheet) planar domain bounded by a $C^\infty$-smooth closed curve $\partial\Omega$. See \cite[Section 4]{JS} for the discussion of simply connected multisheet planar domains. The {\it Dirichlet-to-Neumann operator} of the domain
$$
\Lambda_\Omega:C^\infty(\partial\Omega)\rightarrow C^\infty(\partial\Omega)
$$
is defined by $\Lambda_\Omega f=\frac{\partial u}{\partial\nu}|_{\partial\Omega}$, where $\nu$ is the outward unit normal to $\partial\Omega$ and $u$ is the solution to the Dirichlet problem
$$
\Delta u=0\quad\mbox{\rm in}\quad\Omega,\quad u|_{\partial\Omega}=f.
$$
The Dirichlet-to-Neumann operator is a first order pseudodifferential operator. Moreover, it is a non-negative self-adjoint operator with respect to the $L^2$-product
$$
\l u,v\r=\int\limits_{\partial\Omega}u\bar v\,ds,
$$
where $ds$ is the Euclidean arc length of the curve $\partial\Omega$. In particular, the operator $\Lambda_\Omega$ has a non-negative discrete eigenvalue spectrum
$$
\mbox{\rm Sp}(\Omega)=\{0=\lambda_0(\Omega)<\lambda_1(\Omega)\leq\lambda_2(\Omega)\leq\dots\},
$$
where each eigenvalue is repeated according to its multiplicity. The spectrum is called the {\it Steklov spectrum} of the domain $\Omega$. In particular, for the unit disk
${\mathbb D}=\{(x,y)\mid x^2+y^2\leq1\}$,
$$
\mbox{\rm Sp}({\mathbb D})=\{0=\lambda^0_0<\lambda^0_1\leq\lambda^0_2\leq\dots\}=\{0,1,1,2,2,\dots\}.
$$

Let ${\mathbb S}=\partial{\mathbb D}=\{e^{i\theta}\}\subset{\mathbb C}$ be the unit circle. The Dirichlet-to-Neumann operator of the unit disk will be denoted by $\Lambda:C^\infty({\mathbb S})\rightarrow C^\infty({\mathbb S})$, i.e., $\Lambda=\Lambda_{\mathbb D}$ (this operator was denoted by $\Lambda_e$ in \cite{JS} and \cite{MS}). Given a positive function $a\in C^\infty({\mathbb S})$, the operator $a\Lambda$ has the non-negative discrete eigenvalue spectrum
$$
\mbox{\rm Sp}(a)=\{0=\lambda_0(a)<\lambda_1(a)\leq\lambda_2(a)\leq\dots\}
$$
which is called the {\it Steklov spectrum of the function} $a$ (or of the operator $a\Lambda$).

Two kinds of the Steklov spectrum are related as follows. Given a smooth simply connected planar domain $\Omega$, choose a biholomorphism $\Phi:{\mathbb D}\rightarrow\Omega$ and define the function $0<a\in C^\infty({\mathbb S})$ by $a(z)=|\Phi'(z)|^{-1}\ (z\in{\mathbb S})$. Let $\phi:{\mathbb S}\rightarrow\partial\Omega$ be the restriction of $\Phi$ to ${\mathbb S}$. Then $a\Lambda=\phi^*\Lambda_\Omega\,\phi^{*-1}$ and $\mbox{\rm Sp}(a)=\mbox{\rm Sp}(\Omega)$. See \cite[Section 3]{JS} for details.

The biholomorphism $\Phi$ of the previous paragraph is defined up to a conformal transformation of the disk $\D$, this provides examples of functions with the same Steklov spectrum. Two functions $a,b\in C^\infty({\mathbb S})$ are said to be {\it conformally equivalent}, if there exists a conformal or anticonformal transformation $\Psi$ of the disk ${\mathbb D}$ such that
\begin{equation}
b=\left|\frac{d\psi}{d\theta}\right|^{-1}a\circ\psi,
                                            \label{0.1a}
\end{equation}
where the function $\psi(\theta)$ is defined by $e^{i\psi(\theta)}=\Psi(e^{i\theta})$
($\Psi$ is anticonformal if $\bar\Psi$ is conformal).
If two positive functions $a,b\in C^\infty({\mathbb S})$ are conformally equivalent, then $\mbox{Sp}(a)=\mbox{Sp}(b)$.

Let $D=-i\frac{d}{d\theta}:C^\infty({\mathbb S})\rightarrow C^\infty({\mathbb S})$ be the differentiation with respect to the angle variable.
{\it Steklov zeta-invariants} $Z_m(a)\ (m=1,2,\dots)$ of a function $a\in C^\infty(\S)$ are defined by
\begin{equation}
Z_m(a)=\T[(a\Lambda)^{2m}-(aD)^{2m}].
                                            \label{0.2}
\end{equation}
The expression in brackets is a smoothing operator on $\S$, this fact will be proved below. Remind that every smoothing pseudodifferential operator on a compact manifold has a finite trace. We emphasize that zeta-invariants are well defined for an arbitrary (complex-valued) function $a\in C^\infty(\S)$ although the  spectrum of $a\Lambda$ can be not discrete in the general case. Zeta-invariants are real for a real function $a$, we will mostly study this case. For a positive
function $a\in C^\infty(\S)$, zeta-invariants are uniquely determined by the Steklov spectrum $\mbox{\rm Sp}(a)$. This fact was proved by J.~Edward \cite{E} (without using the term ``zeta-invariants''), see also \cite{MS}.

For a function $u$ on the circle $\S=\{e^{i\theta}\mid\theta\in\R\}$, we write $u(\theta)$ instead of $u(e^{i\theta})$.
Fourier coefficients of $u\in C^\infty(\S)$ are denoted by ${\hat u}_n$, i.e., $u(\theta)=\sum_{n\in\Z}{\hat u}_ne^{in\theta}$. Edward \cite{E} obtained the formula

\begin{equation}
Z_1(a)=\frac{2}{3}\sum\limits_{n=2}^\infty(n^3-n)\, |{\hat a}_n|^2
                                               \label{0.3}
\end{equation}
for a real function $a\in C^\infty(\S)$. E.~Malkovich and V.~Sharafutdinov \cite{MS} generalized the formula to all zeta-invariants: $Z_m(a)$ is expressed by some $2m$-form in Fourier coefficients ${\hat a}_n$. Unfortunately, the latter formula is too complicated to be useful for deriving theoretical results. On the other hand,
Malkovich -- Sharafutdinov's formula is very easy for computerization. Thus, unlike the hard problem of calculating Steklov eigenvalues, zeta-invariants can be easily computed.

The main result of the present paper is the following

\begin{theorem} \label{Th0.1}
Given an integer $m \ge 1$ and a real function $a\in C^{\infty}(\S)$, set $b=a^m$.
The estimate
\begin{equation}
Z_m(a)\ge c_m\sum_{n=m+1}^\infty n^{2m+1}|{\hat b}_n|^2
                                          \label{0.4}
\end{equation}
holds with some positive constant $c_m$ independent of $a$. In particular, $Z_m(a)\ge0$ for every integer $m\ge1$ and for every real function $a\in C^{\infty}(\S)$.
\end{theorem}

For $m=1$, estimate \eqref{0.4} follows from Edward's formula \eqref{0.3}. Theorem \ref{Th0.1} was first conjectured as a result of a lot of numerical experiments based on Malkovich -- Sharafut\-di\-nov's formula, we are grateful to E.~Malkovich for his help with computer calculations. For a positive function $a$, the statement
$Z_m(a)\ge0$ follows from a more general inequality proved by the authors \cite{JS2}.

The question of describing the null space of a zeta-invariant was posed in \cite{MS}. The question is closely related to the above-defined conformal equivalence of functions. The following theorem gives the full answer to the question for real functions.

\begin{theorem}  \label{Th0.2}
For every integer $m\geq1$ and for every real function $a\in C^\infty(\S)$,
$Z_m(a)=0$
if and only if the function $a$ is of the form
\begin{equation}
a(\theta)={\hat a}_0+2\Re({\hat a}_1 e^{i\theta})
\quad\mbox{with some }{\hat a}_0\in\R\mbox{ and }{\hat a}_1\in\C.
                              \label{0.5}
\end{equation}
\end{theorem}

Observe that \eqref{0.5} holds if and only if $a$ is conformally equivalent to a constant function. The latter fact is not used in our proof. Again, the statement of the theorem follows from Edward's formula \eqref{0.3} in the case of $m=1$. The ``if'' statement of Theorem \ref{Th0.2} is proved in \cite[Section 6]{MS}. Moreover, the following more general statement is proved there. For every integer $m\ge1$, $Z_m(a)=0$ if $a\in C^\infty(\S)$ is a (complex-valued) function satisfying ${\hat a}_n=0$ for $|n|>1$. Our proof is independent of the latter statement.

The inverse problem of recovering a function $0<a\in C^\infty(\S)$ from the Steklov spectrum $\mbox{\rm Sp}(a)$ seems to be very difficult. It makes sense to start with easier questions of the following kind. Given $0<a\in C^\infty(\S)$, set
$$
{\mathcal A}=\{0<\tilde a\in C^\infty(\S)\mid\tilde a\mbox{ is conformally equivalent to } a\}.
$$
How far off ${\mathcal A}$ can be a function $0<b\in C^\infty(\S)$ satisfying $\mbox{\rm Sp}(b)=\mbox{\rm Sp}(a)$? In this direction, we prove the following compactness theorem.

\begin{theorem}  \label{Th0.3}
Let $a_n\in C^\infty(\S)\ (n=1,2,\dots)$ be a sequence of positive functions such that the Steklov spectrum $\mbox{\rm Sp}(a_n)$ is independent of $n$.
There exists a subsequence $a_{n_k}\ (k=1,2,\dots)$ such that every $a_{n_k}$ is conformally equivalent to some $b_k\in C^\infty(\S)$ and the sequence
$b_k\ (k=1,2,\dots)$ converges to a positive function $b\in C^\infty(\S)$ in the $C^\infty$-topology, i.e., $D^{\ell}b_k\rightarrow D^{\ell}b$ as $k\rightarrow\infty$ uniformly on $\S$ for every $\ell\in\N$.
\end{theorem}

The theorem gives the positive answer to Edward's question \cite{E2} who has proved the corresponding pre-compactness theorem in the Sobolev $H^s$-topology for $s<5/2$. In his proof, Edward uses first two zeta-invariants and the values $\zeta_a(-1),\zeta_a(-3)$ of the zeta function. Our proof follows the same line with using estimate \eqref{0.4} for all zeta-invariants.

The paper is organized as follows.

Let $\Psi({\mathbb S})$ be the algebra of all pseudodifferential operators on ${\mathbb S}$ considered as an algebra over ${\mathbb C}$.  In Section 2, we consider the subalgebra ${\mathbb C}[L,H]$ of $\Psi({\mathbb S})$ generated by the Hilbert transform $H$ and a general self-adjoint operator
$L\in\Psi({\mathbb S})$ that commutes with $H$ up to a smoothing operator. We prove a number of statements on traces of some smoothing operators belonging to
${\mathbb C}[L,H]$. Theorems \ref{Th0.1} and \ref{Th0.2} are the most important examples of such statements. Indeed, definition \eqref{0.2} can be written in terms of the operator
\begin{equation}
L=\Lambda^{1/2}a\Lambda^{1/2}
                              \label{0.6}
\end{equation}
(we denote the operator of multiplication by a function $a$ by the same letter $a$) as
\begin{equation}
Z_m(a)=\T[L^{2m}-(LH)^{2m}].
                                            \label{0.7}
\end{equation}
The commutator $[L,H]$ is a smoothing operator as is shown below.
Proofs of Theorems \ref{Th0.1} and \ref{Th0.2} are presented in Sections 3 and 4 respectively. Section 5 contains the proof of Theorem \ref{Th0.3} and an interpretation of the theorem in terms of a Steklov isospectral family of planar domains.

The proof of Theorem \ref{Th0.3} is independent of Theorem \ref{Th0.2}. If a reader is not interested in Theorem \ref{Th0.2}, he/she does not need to read Theorem \ref{Th1.2} and the rest of Section 2 as well as Section 4. The proof of Lemma \ref{L1.1} is presented in Appendix. If a reader either knows a better proof of the equality $\T[AB-C]=\T[BA-C]$ in the setting of Lemma \ref{L1.1} or believes that the equality is always true, he/she can do not look into Appendix.

\section{Algebra ${\mathbb C}[L,H]$}

Let $\S=\{e^{i\theta}\mid\theta\in{\mathbb R}\}$ be the unit circle. For a function $u\in C(\S)$, we write $u(\theta)$ instead of $u(e^{i\theta})$.
The $L^2$-product is denoted by
$$
\l u,v\r=\frac{1}{2\pi}\int\limits_0^{2\pi}u(\theta)\bar v(\theta)\,d\theta
$$
and the corresponding norm is $\|\cdot\|$. The factor $(2\pi)^{-1}$ is included to make $\{e^{in\theta}\}_{n\in\Z}$ the orthonormal basis of $L^2(\S)$.

We define the Hilbert transform $H:C^\infty(\S)\rightarrow C^\infty(\S)$ by
\begin{equation}
He^{in\theta}=
\left\lbrace
\begin{array}{l}
e^{in\theta}\quad \textrm{if}\quad n\geq 0,\\
-e^{in\theta}\quad \textrm{if}\quad n<0.
\end{array}
\right.
                                      \label{1.0}
\end{equation}
Our definition is slightly different of the standard one: $He^{in\theta}=0$ for $n=0$ according to the standard definition.
Observe that $H$ is a zero order pseudodifferential operator. Additionally, it is a unitary operator on $L^2(\S)$  satisfying $H^2=I$.

Let $\Psi({\mathbb S})$ be the algebra of all pseudodifferential operators on ${\mathbb S}$ considered as an algebra over ${\mathbb C}$. We fix an operator $L\in\Psi({\mathbb S})$ and consider the subalgebra ${\mathbb C}[L,H]$ of $\Psi({\mathbb S})$ generated by $L$ and $H$. From the algebraic viewpoint, ${\mathbb C}[L,H]$ is the algebra of polynomials in two variables $(L,H)$.
Every monomial of the algebra ${\mathbb C}[L,H]$
can be written in the form $\lambda A\ (\lambda\in{\mathbb C})$, where
\begin{equation}
A=L^{j_1}H\dots L^{j_s}H\quad ( 0\leq j_\alpha\in{\mathbb N}\quad \mbox{for}\quad \alpha=1,\dots,s).
                                      \label{1.1d}
\end{equation}
The sum $j_1+\dots+j_s$ is called the degree of the monomial in $L$ while $s$ is the degree of the monomial in $H$.

Assume now that the commutator $[L,H]$ is a smoothing operator. Then we can commute, up to a smoothing operator, factors of product (\ref{1.1d}). In this way (\ref{1.1d}) can be reduced, up to a smoothing operator, either to $L^{j_1+\dots+j_s}$ (if the degree of $A$ in $H$ is even) or to $L^{j_1+\dots+j_s}H$ (if the degree of $A$ in $H$ is odd). In particular, the following statement is valid.

\begin{lemma} \label{L1.0}
If two products $A_1$ and $A_2$ of form (\ref{1.1d}) are of the same degree in $L$ and their degrees in $H$ are of the same evenness, then the difference $A_1-A_2$ is a smoothing and hence finite trace operator.
\end{lemma}

An integer $m\geq1$ will be fixed till the end of the current section. The dependence of different quantities on $m$ will not be designated explicitly.
The proof of the following lemma is presented in Appendix.

\begin{lemma} \label{L1.1}
Let $L$ be a self-adjoint pseudodifferential operator on $\S$ such that the commutator $[L,H]$ is a smoothing operator.
Let $A_1$ and $A_2$ be two products of form (\ref{1.1d}) whose degrees in $H$ are of the same evenness. Assume that $m_1+m_2=2m$, where $m_i$ is the degree of $A_i$ in $L$ ($i=1,2$). Then
\begin{equation}
\T\,[A_1A_2-(LH)^{2m}]=\T\,[A_2A_1-(HL)^{2m}].
                                      \label{1.2}
\end{equation}
\end{lemma}

Let us mention two partial cases of (\ref{1.2}). For $A_1=H$ and $A_2=L^{2m}H$, (\ref{1.2}) gives
\begin{equation}
\T\,[HL^{2m}H-(LH)^{2m}]=\T\,[L^{2m}-(LH)^{2m}].
                                      \label{1.3}
\end{equation}
For  $A_1=H$ and $A_2=H\prod_{\ell=1}^{2s}(L^{j_\ell}H)$ with $j_1+\dots+j_{2s}=2m$, (\ref{1.2}) gives
$$
\T\,\Big[\prod\limits_{\ell=1}^{2s}(L^{j_\ell}H)-(LH)^{2m}\Big]=\T\,\Big[\prod\limits_{\ell=1}^{2s}(HL^{j_\ell})-(LH)^{2m}\Big].
$$

\bigskip

For every $\ell\in \N$ satisfying $1\le \ell\le  m$, let $\Delta_\ell$ be the set of sequences $\delta=(\delta_1,\dots,\delta_\ell)$ of integers satisfying $1\leq \delta_i\leq m\ (1\leq i\leq\ell)$ and $\delta_1+\dots+\delta_\ell=m$.

Let $L$ be a self-adjoint pseudodifferential operator on $\S$ such that the commutator $[L,H]$ is a smoothing operator.
For $\delta=(\delta_1,\dots,\delta_\ell)\in \Delta_\ell$, we set
\begin{equation}
G_\delta=(L^{\delta_1}H)(L^{\delta_2}H)\dots(L^{\delta_\ell}H).
                               \label{1.6}
\end{equation}
The trace
$\T\,[G_\delta^*G_\delta-(LH)^{2m}]$ is real. Indeed, on using the trigonometric basis, we have
$$
\T\,[G_\delta^*G_\delta-(LH)^{2m}]=\sum\limits_{n\in{\mathbb Z}}\Big(\l G_\delta^*G_\delta e^{in\theta},e^{in\theta}\r
-\l(LH)^{2m} e^{in\theta},e^{in\theta}\r\Big).
$$
All summands of the series are real. Indeed, since $He^{in\theta}=\pm e^{in\theta}$,
$$
\begin{aligned}
\l G_\delta^*G_\delta e^{in\theta},e^{in\theta}\r-\l(LH)^{2m} e^{in\theta},e^{in\theta}\r
&=\|G_\delta e^{in\theta}\|^2-\l(LH)^{2m-1}LH e^{in\theta},e^{in\theta}\r\\
&=\|G_\delta e^{in\theta}\|^2\mp\l(LH)^{2m-1}L e^{in\theta},e^{in\theta}\r.
\end{aligned}
$$
The right-hand side is real because $(LH)^{2m-1}L$ is a self-adjoint operator.
Observe also that
$$
\T\,[G_\delta^*G_\delta-(LH)^{2m}]=\T\,[G_\delta G_\delta^*-(LH)^{2m}]
$$
by Lemma \ref{L1.1}.

Define the function $\varphi:\{1,\dots,m\}\rightarrow{\mathbb R}$ by
\begin{equation}
\varphi(\ell)
=\max_{\delta\in \Delta_\ell}\T\,[G_\delta^*G_\delta-(LH)^{2m}]\quad(1\leq\ell\leq m).
                                       \label{1.8}
\end{equation}

\begin{theorem} \label{Th1.1}
Let $L$ be a self-adjoint pseudodifferential operator on $\S$ such that the commutator $[L,H]$ is a smoothing operator and let the function $\varphi$ be defined by (\ref{1.8}).
Then $\varphi$ is a non-increasing and non-negative function. In particular,
\begin{equation}
\T\,[L^{2m}-(LH)^{2m}]= \varphi(1)\ge \varphi(m)=\T\,\big[H(LH)^{m-1}L^2H(LH)^{m-1}-(LH)^{2m}\big]\ge 0.
                                                \label{1.9}
\end{equation}
\end{theorem}

To prove the theorem, we need the following

\begin{lemma} \label{L1.3}
Let $A$, $B$, and $C$ be pseudodifferential operators on $\S$ such that $C$ is self-adjoint and $AA^*-CH$, $B^*B-CH$, and $AB-CH$ are finite trace operators. Then
\begin{equation}
\Re\, \big(\T\,[AB-CH]\big)\le {1\over 2}\T\,[AA^*-CH]+{1\over 2}\T\,[B^*B-CH].
                              \label{1.10}
\end{equation}
The equality on {\rm (\ref{1.10})} holds if and only if $A^*=B$.
\end{lemma}

\begin{proof}
For $n\in \Z$,
$$
\l(AB-CH)e^{in\theta},e^{in\theta}\r=\l Be^{in\theta},A^*e^{in\theta}\r-\l CHe^{in\theta},e^{in\theta}\r.
$$
By the Schwartz inequality,
\begin{equation}
\Re\,\big( \l Be^{in\theta}, A^*e^{in\theta}\r\big)\le \|Be^{in\theta}\|\,\|A^*e^{in\theta}\|
\le{1\over 2}\|A^*e^{in\theta}\|^2+{1\over 2}\|Be^{in\theta }\|^2.
                              \label{1.11}
\end{equation}
Taking into account that $\l CHe^{in\theta},e^{in\theta}\r=\pm \l Ce^{in\theta},e^{in\theta}\r$ is real, we derive from two last formulas
$$
\begin{aligned}
\Re\big(\l(AB-CH)e^{in\theta},&e^{in\theta}\r\big)\\&\le {1\over 2}\big(\|A^*e^{in\theta}\|^2-\l CHe^{in\theta},e^{in\theta}\r\big)
+\frac{1}{2}\big(\|Be^{in\theta}\|^2-\l CHe^{in\theta},e^{in\theta}\r\big)\\
&= {1\over 2}\big\langle (AA^*-CH)e^{in\theta},e^{in\theta}\big\rangle
+\frac{1}{2}\big\langle(B^*B-CH)e^{in\theta},e^{in\theta}\big\rangle.
\end{aligned}
$$
Summing these inequalities over $n\in \Z$, we arrive to \eqref{1.10}.

Equality in \eqref{1.10} holds if and only if both inequalities on \eqref{1.11} become equalities for every $n\in \Z$, i.e., if and only if $A^*e^{in\theta}=Be^{in\theta}$ for any $n\in \Z$. Hence equality in \eqref{1.10} holds if and only if $A^*=B$.
\end{proof}

\begin{proof}[Proof of Theorem \ref{Th1.1}.]
Define $\omega_\ell\in\Delta_\ell\ (1\leq\ell\leq m)$ by
\begin{equation}
\omega_\ell=(m-\ell+1,\underbrace{1,\dots,1}_{\ell-1}).
                              \label{1.12}
\end{equation}
In particular, $\omega_1=(m)$ and $\omega_m=(\underbrace{1,\dots,1}_{m})$.

Obviously, $G_{\omega_1}=L^mH$ and
$$
\varphi(1)=\T\,[G_{\omega_1}^*G_{\omega_1}-(LH)^{2m}]=\T\,[HL^{2m}H-(LH)^{2m}].
$$
With the help of \eqref{1.3}, this gives $\varphi(1)=\T\,[L^{2m}-(LH)^{2m}]$. We have thus proved the first equality on \eqref{1.9}.

Obviously,
\begin{equation}
G_{\omega_m}=(LH)^m
                                                 \label{1.12a}
\end{equation}
and
\begin{equation}
\begin{aligned}
\varphi(m)=\T\,\big[G_{\omega_m}^*G_{\omega_m}-(LH)^{2m}\big]&=\T\,\big[(HL)^m(LH)^m-(LH)^{2m}\big]\\&=\T\,\big[H(LH)^{m-1}L^2H(LH)^{m-1}-(LH)^{2m}\big].
\end{aligned}
                                                 \label{1.13}
\end{equation}
This proves the last equality on \eqref{1.9}.

Next, we prove that $\varphi(m)\ge 0$. By \eqref{1.13},
\begin{equation}
\varphi(m)=\T\,\big[G_{\omega_m}^*G_{\omega_m}-(LH)^{2m}\big]
=\sum\limits_{n\in \Z}\big(\|G_{\omega_m} e^{in\theta}\|^2-\l (LH)^{2m}e^{in\theta} ,e^{in\theta}\r\big).
                                                 \label{1.14}
\end{equation}
For every $n\in \Z$, we obtain with the help of \eqref{1.12a}
$$
\l (LH)^{2m}e^{in\theta} ,e^{in\theta}\r=
\l G_{\omega_m}^2e^{in\theta} ,e^{in\theta}\r
=\l G_{\omega_m}e^{in\theta} ,G_{\omega_m}^*e^{in\theta}\r.
$$
As is seen from \eqref{1.12a}, $G_{\omega_m}^*=HG_{\omega_m}H$ and the previous formula takes the form
$$
\l (LH)^{2m}e^{in\theta} ,e^{in\theta}\r=
\l G_{\omega_m}e^{in\theta} ,HG_{\omega_m}He^{in\theta}\r.
$$
This implies with the help of the Schwartz inequality
\begin{equation}
\l (LH)^{2m}e^{in\theta} ,e^{in\theta}\r=
\l HG_{\omega_m}e^{in\theta} ,G_{\omega_m}He^{in\theta}\r\leq\|HG_{\omega_m}e^{in\theta}\|\,\|G_{\omega_m}He^{in\theta}\|.
                                                 \label{1.15}
\end{equation}
Since $He^{in\theta}=\pm e^{in\theta}$ and $H$ is a unitary operator, \eqref{1.15} is equivalent to
$$
\l (LH)^{2m}e^{in\theta} ,e^{in\theta}\r\leq\|G_{\omega_m}e^{in\theta}\|^2.
$$
Combining this with \eqref{1.14}, we obtain $\varphi(m)\ge 0$.

Finally, we prove that $\varphi(\ell)\le \varphi(\ell-1)$ for $2\le \ell\le m$.
Let $\delta=(\delta_1,\ldots, \delta_\ell)\in \Delta_\ell$ be such that
$$
\varphi(\ell)=\T\,[G_\delta^*G_\delta-(LH)^{2m}].
$$
This can be rewritten as
$$
\varphi(\ell)=\T\,\big[(HL^{\delta_\ell})\dots (HL^{\delta_2})(HL^{2\delta_1})H(L^{\delta_2}H)\dots (L^{\delta_\ell}H)-(LH)^{2m}\big].
$$
By Lemma \ref{L1.1}, the last factor $L^{\delta_\ell}H$ of the first product can be moved to the first position. Applying also $H^2=I$, we obtain
\begin{equation}
\varphi(\ell)=\T\,\Big[L^{2\delta_\ell}H\prod\limits_{i=2}^{\ell-1}(L^{\delta_{\ell-i+1}}H)L^{2\delta_1}H
\prod\limits_{i=2}^{\ell-1}(L^{\delta_i}H)
-(LH)^{2m}\Big].
                                                 \label{1.16}
\end{equation}
Again by Lemma \ref{L1.1}, this can be equivalently written as
$$
\varphi(\ell)=\T\,\Big[L^{2\delta_1}H\prod\limits_{i=2}^{\ell-1}(L^{\delta_i}H)L^{2\delta_\ell}\prod\limits_{i=2}^{\ell-1}(L^{\delta_{l-i+1}}H)-(LH)^{2m}\Big].
$$
Therefore we can assume without loss of generality that $\delta_1\ge \delta_\ell$ in (\ref{1.16}) (the product $\Pi_{i=2}^{\ell-1}$ does not appear in (\ref{1.16}) when $\ell=2$).

We write (\ref{1.16}) in the form
\begin{equation}
\varphi(\ell)=\T\,[AB-CH],
                                                 \label{1.17}
\end{equation}
where
\begin{equation}
A=L^{2\delta_\ell}H\prod\limits_{i=2}^{\ell-1}(L^{\delta_{\ell-i+1}}H)L^{\delta_1-\delta_\ell},\quad
B=L^{\delta_1+\delta_\ell}H\prod\limits_{i=2}^{\ell-1}(L^{\delta_i}H),\quad C=(LH)^{2m-1}L.
                                                 \label{1.18}
\end{equation}
The operators $A,B,C$ satisfy hypotheses of Lemma \ref{L1.3} and inequality (\ref{1.10}) holds. From (\ref{1.17}) and (\ref{1.10}),
$$
\varphi(\ell)\le {1\over 2}\T\,[AA^*-CH]+{1\over 2}\T\,[B^*B-CH].
$$
With the help of Lemma \ref{L1.1}, this can be written in the form
\begin{equation}
\begin{aligned}
\varphi(\ell)&\le {1\over 2}\T\,[HAA^*H-CH]+{1\over 2}\T\,[B^*B-CH]\\
&={1\over 2}\T\,[(A^*H)^*(A^*H)-CH]+{1\over 2}\T\,[B^*B-CH].
\end{aligned}
                                                 \label{1.19}
\end{equation}

Assume first that $\delta_1>\delta_\ell$. Comparing (\ref{1.6}) and (\ref{1.18}), we see that
\begin{equation}
A^*H=G_{(\delta_1-\delta_\ell,\delta_2,\dots,\delta_{\ell-1},2\delta_\ell)},\quad B=G_{(\delta_1+\delta_\ell,\delta_2,\dots,\delta_{\ell-1})},\quad CH=(LH)^{2m}.
                                                 \label{1.20}
\end{equation}
Now, (\ref{1.19}) takes the form
\begin{equation}
\begin{aligned}
\varphi(\ell)&\le{1\over 2}\T\,\big[G_{(\delta_1-\delta_\ell,\delta_2,\dots,\delta_{\ell-1},2\delta_\ell)}^*G_{(\delta_1-\delta_\ell,\delta_2,\dots,\delta_{\ell-1},2\delta_\ell)}-(LH)^{2m}\big]\\
&+{1\over 2}\T\,\big[G_{(\delta_1+\delta_\ell,\delta_2,\dots,\delta_{\ell-1})}^*G_{(\delta_1+\delta_\ell,\delta_2,\dots,\delta_{\ell-1})}-(LH)^{2m}\big].
\end{aligned}
                                                 \label{1.21}
\end{equation}
Here, $G_{(\delta_1-\delta_\ell,\delta_2,\dots,\delta_{\ell-1},2\delta_\ell)}$ and $G_{(\delta_1+\delta_\ell,\delta_2,\dots,\delta_{\ell-1})}$ should be replaced by
$G_{(\delta_1-\delta_2,2\delta_2)}$ and $G_{(\delta_1+\delta_2)}$ respectively in the case of $l=2$. By the definition of the function $\varphi$,
$$
\begin{aligned}
&\T\,\big[G_{(\delta_1-\delta_\ell,\delta_2,\dots,\delta_{\ell-1},2\delta_\ell)}^*G_{(\delta_1-\delta_\ell,\delta_2,\dots,\delta_{\ell-1},2\delta_\ell)}-(LH)^{2m}\big]\leq\varphi(\ell),\\
&\T\,\big[G_{(\delta_1+\delta_\ell,\delta_2,\dots,\delta_{\ell-1})}^*G_{(\delta_1+\delta_\ell,\delta_2,\dots,\delta_{\ell-1})}-(LH)^{2m}\big]\leq\varphi(\ell-1).
\end{aligned}
$$
Therefore (\ref{1.21}) implies the desired inequality $\varphi(\ell)\leq\varphi(\ell-1)$.

Finally, we consider the case of $\delta_1=\delta_\ell$. In this case, we have instead of (\ref{1.20})
$$
A^*H=HG_{(\delta_2,\dots,\delta_{\ell-1},2\delta_\ell)},\quad B=G_{(\delta_1+\delta_\ell,\delta_2,\dots,\delta_{\ell-1})},\quad CH=(LH)^{2m}
$$
and (\ref{1.19}) takes the form
\begin{equation}
\begin{aligned}
\varphi(\ell)&\le{1\over 2}\T\,\big[G_{(\delta_2,\dots,\delta_{\ell-1},2\delta_\ell)}^*G_{(\delta_2,\dots,\delta_{\ell-1},2\delta_\ell)}-(LH)^{2m}\big]\\
&+{1\over 2}\T\,\big[G_{(\delta_1+\delta_\ell,\delta_2,\dots,\delta_{\ell-1})}^*G_{(\delta_1+\delta_\ell,\delta_2,\dots,\delta_{\ell-1})}-(LH)^{2m}\big].
\end{aligned}
                                                 \label{1.22}
\end{equation}
Here $G_{(\delta_2,\dots,\delta_{\ell-1},2\delta_\ell)}$ and $G_{(\delta_1+\delta_\ell,\delta_2,\dots,\delta_{\ell-1})}$ should be replaced by
$G_{(2\delta_\ell)}$ and $G_{(\delta_1+\delta_\ell)}$ respectively in the case of $l=2$.
By the definition of the function $\varphi$,
$$
\begin{aligned}
&\T\,\big[G_{(\delta_2,\dots,\delta_{\ell-1},2\delta_\ell)}^*G_{(\delta_2,\dots,\delta_{\ell-1},2\delta_\ell)}-(LH)^{2m}\big]\leq\varphi(\ell-1),\\
&\T\,\big[G_{(\delta_1+\delta_\ell,\delta_2,\dots,\delta_{\ell-1})}^*G_{(\delta_1+\delta_\ell,\delta_2,\dots,\delta_{\ell-1})}-(LH)^{2m}\big]\leq\varphi(\ell-1).
\end{aligned}
$$
Therefore (\ref{1.22}) again implies the desired inequality $\varphi(\ell)\leq\varphi(\ell-1)$.
\end{proof}

\bigskip

We have the following characterization of the equality in Theorem \ref{Th1.1}.

\begin{theorem} \label{Th1.2}
Let $L$ be a self-adjoint pseudodifferential operator on $\S$ such that $[L,H]$ is a smoothing operator. The equality
\begin{equation}
\T[L^{2m}-(LH)^{2m}]=0
                                           \label{1.23}
\end{equation}
holds if and only if
\begin{equation}
(LH)^m=(HL)^m\quad\textrm{and}\quad L^2H=HL^2.
                                     \label{1.24}
\end{equation}
\end{theorem}

The proof of the theorem is  based on the following

\begin{lemma} \label{L1.4}
Let $L$ be a self-adjoint pseudodifferential operator on $\S$ such that the commutator $[L,H]$ is smoothing. Equality (\ref{1.23})
holds if and only if
\begin{equation}
(LH)^m=(HL)^m
                                    \label{1.25}
\end{equation}
and
\begin{equation}
L^{\ell+1}H(LH)^{m-\ell-1}=L^{\ell-1}H(LH)^{m-\ell-1}L^2\quad\textrm{for}\quad 1\le \ell\le m-1.
                                    \label{1.26}
\end{equation}
Condition (\ref{1.26}) is absent in the case of $m=1$.
\end{lemma}

\begin{proof}
By Theorem \ref{Th1.1}, \eqref{1.23} holds if and only if
\begin{equation}
\varphi(\ell)=0\quad\mbox{for}\quad 1\le\ell\leq m.
                                         \label{1.27}
\end{equation}

As is seen from \eqref{1.14} and \eqref{1.15}, the equality  $\varphi(m)=0$ is equivalent to the following statement: the equality
$$
\l H G_{\omega_m}e^{in\theta}, G_{\omega_m}He^{in\theta}\r
=\|HG_{\omega_m}e^{in\theta}\|\,\|G_{\omega_m}He^{in\theta}\|
$$
holds for every $n\in \Z$. Besides this, the norms $\|HG_{\omega_m}e^{in\theta}\|$ and $\|G_{\omega_m}He^{in\theta}\|$ coincide because $He^{in\theta}=\pm e^{in\theta}$ and $H$ is a unitary operator. Therefore $\varphi(m)=0$ if and only if
$H G_{\omega_m}e^{in\theta}= G_{\omega_m}He^{in\theta}\ (n\in \Z)$, i.e., if and only if
$$
HG_{\omega_m}=G_{\omega_m}H.
$$
On using \eqref{1.12a} and $H^2=I$, we see that the latter equality is equivalent to \eqref{1.25}.
In particular, we have proved the theorem in the case of $m=1$.

Assume $m\ge 2$ for the rest of the proof. Replacing $\ell$ with $m-\ell$, we rewrite \eqref{1.26} in the equivalent form
\begin{equation}
L^{m-\ell+1}H(LH)^{\ell-1}=L^{m-\ell-1}H(LH)^{\ell-1}L^2\quad\textrm{for}\quad 1\le \ell\le m-1.
                                    \label{1.27a}
\end{equation}

We first prove the ``only if'' statement. Assume \eqref{1.25} to be valid.
Let $\omega_\ell\in\Delta_\ell$ be defined by \eqref{1.12}.
By induction in $m-\ell$, we will prove the validity of \eqref{1.27a} and of the equality
\begin{equation}
0=\varphi(\ell)=\T[G_{\omega_\ell}^*G_{\omega_\ell}-(LH)^{2m}]\quad\mbox{for}\quad 1\le l\le m.
                                                    \label{1.28}
\end{equation}
For $\ell=m$, \eqref{1.28} holds by \eqref{1.14} and \eqref{1.27}.

Now, we prove (\ref{1.27a}) and (\ref{1.28}) for $l=m-1$. By \eqref{1.9},
$$
\varphi(m)=\T\,\big[H(LH)^{m-1}L^2H(LH)^{m-1}-(LH)^{2m}\big].
$$
By Lemma \ref{L1.1}, we can transpose the factors $H(LH)^{m-1}$ and $L^2H(LH)^{m-1}$ on the right-hand side
$$
\begin{aligned}
\varphi(m)&=\T\,\big[L^2H(LH)^{m-1}H(LH)^{m-1}-(LH)^{2m}\big]\\
&=\T\,\big[L^2H(LH)^{m-2}L^2H(LH)^{m-2}-(LH)^{2m}\big].
\end{aligned}
$$
This can be written in the form
$$
\varphi(m)=\T\,[A^2-(LH)^{2m}],
$$
where
\begin{equation}
A=L^2H(LH)^{m-2}=G_{\omega_{m-1}}.
                                       \label{1.31}
\end{equation}
The latter equality follows from (\ref{1.6}) and (\ref{1.12}).
Operators $A=B$ and $C=(LH)^{2m-1}L$ satisfy hypotheses of Lemma \ref{L1.3}. This is checked quite similarly to the corresponding check after formula \eqref{1.18}. Applying Lemma \ref{L1.3}, we obtain
\begin{equation}
\begin{aligned}
0=\varphi(m)&\le{1\over 2}\T\,\big[G_{\omega_{m-1}}^*G_{\omega_{m-1}}-(LH)^{2m}\big]+{1\over 2}\T\,\big[G_{\omega_{m-1}}G_{\omega_{m-1}}^*-(LH)^{2m}\big]\\
&=\T\big[G_{\omega_{m-1}}^*G_{\omega_{m-1}}-(LH)^{2m}\big]
\le \varphi(m-1)=0.
\end{aligned}
                                       \label{1.32}
\end{equation}
Thus, we have the equality in Lemma \ref{L1.3} which provides the following equality
$$
L^2H(LH)^{m-2}=A=A^*=H(LH)^{m-2}L^2.
$$
Hence \eqref{1.27a}--\eqref{1.28} holds for $\ell=m-1$.

Now, we are doing the induction step. Assume  \eqref{1.27a}--\eqref{1.28} to be valid for $\ell=s$ with some $s$ satisfying
$2\le s\le m-1$. We are going to prove \eqref{1.27a}--\eqref{1.28} for $\ell=s-1$. To this end we set
\begin{equation}
A=L^2H(LH)^{s-2}L^{m-s},\quad B=L^{m-s+2}H(LH)^{s-2},\quad C=(LH)^{2m-1}L
                                         \label{1.33a}
\end{equation}
and apply Lemma \ref{L1.3}
\begin{equation}
\T\,[AB-(LH)^{2m}]\le {1\over 2}\T\,[AA^*-(LH)^{2m}]+{1\over 2}\T\,[B^*B-(LH)^{2m}].
                              \label{1.34}
\end{equation}
On using definitions \eqref{1.6}, \eqref{1.12} and Lemma \ref{L1.1}, we easily derive from \eqref{1.33a}
\begin{equation}
\begin{aligned}
\T\,[AB-(LH)^{2m}]&=\T\big[G_{\omega_s}^*G_{\omega_s}-(LH)^{2m}\big],\\
\T\,[AA^*-(LH)^{2m}]&=
\T\big[G_{(m-s,\underbrace{\scriptstyle 1,\dots,1}_{s-2},2)}^*G_{(m-s,\underbrace{\scriptstyle 1,\dots,1}_{s-2},2)}-(LH)^{2m}\big],\\
\T\,[B^*B-(LH)^{2m}]&=\T\big[G_{\omega_{s-1}}^*G_{\omega_{s-1}}-(LH)^{2m}\big].
\end{aligned}
                              \label{1.34a}
\end{equation}
By the induction hypothesis,
$$
\T\,[AB-(LH)^{2m}]=\T\big[G_{\omega_s}^*G_{\omega_s}-(LH)^{2m}\big]=0
$$
and, by the definition of $\varphi$,
$$
\begin{aligned}
\T\,[AA^*&-(LH)^{2m}]+\T\,[B^*B-(LH)^{2m}]\\
&=\T\big[G_{(m-s,\underbrace{\scriptstyle 1,\dots,1}_{s-2},2)}^*G_{(m-s,\underbrace{\scriptstyle 1,\dots,1}_{s-2},2)}-(LH)^{2m}\big]
+\T\big[G_{\omega_{s-1}}^*G_{\omega_{s-1}}-(LH)^{2m}\big]\\
&\leq\varphi(s)+\varphi(s-1)=0.
\end{aligned}
$$
Thus, we have actually the equality in \eqref{1.34}. By Lemma \ref{L1.3}, this means that
$$
L^{m-s}H(LH)^{s-2}L^2=A^*=B=L^{m-s+2}H(LH)^{s-2}.
$$
This proves \eqref{1.27a} for $\ell=s-1$. Two traces on the right-hand side of \eqref{1.34} coincide because $A^*=B$. In other words, two traces on the left-hand side of the equality
$$
\T\big[G_{(m-s,\underbrace{\scriptstyle 1,\dots,1}_{s-2},2)}^*G_{(m-s,\underbrace{\scriptstyle 1,\dots,1}_{s-2},2)}-(LH)^{2m}\big]
+\T\big[G_{\omega_{s-1}}^*G_{\omega_{s-1}}-(LH)^{2m}\big]=0
$$
coincide. Hence,
$$
\T\big[G_{\omega_{s-1}}^*G_{\omega_{s-1}}-(LH)^{2m}\big]=0.
$$
This proves \eqref{1.28} for $\ell=s-1$. The induction step is completed.

\bigskip

Now, we prove the ``if'' statement. Assume \eqref{1.25}--\eqref{1.26} to be valid. We are going to prove by induction in $m-\ell$ that
\begin{equation}
\T[G_{\omega_\ell}^*G_{\omega_\ell}-(LH)^{2m}]=0\quad\mbox{for}\quad 1\le l\le m.
                                                    \label{1.35}
\end{equation}
For $\ell=1$, this gives \eqref{1.23} in virtue of \eqref{1.3} and of $G_{\omega_1}=L^{m}H$.

As we have shown at the beginning of the proof, \eqref{1.25} holds if and only if $\varphi(m)=0$. This implies the validity of \eqref{1.35} for $\ell=m$ (see \eqref{1.13}).

Let the operator $A$ be defined by \eqref{1.31}. Setting $\ell=1$ in \eqref{1.26}, we see that $A=A^*$. With the help of Lemma \ref{L1.3}, this implies that the first inequality on \eqref{1.32} is actually the equality, i.e.,
$$
\T\big[G_{\omega_{m-1}}^*G_{\omega_{m-1}}-(LH)^{2m}\big]=\varphi(m)=0.
$$
This proves \eqref{1.35} for $\ell=m-1$.

Now, we are doing the induction step. Assume  \eqref{1.35} to be valid for $\ell=s$ with some $s$ satisfying
$2\le s\le m-1$. Define operators $A$ and $B$ by \eqref{1.33a}.
Setting $\ell=m-s+1$ in \eqref{1.26}, we see that $A^*=B$. Therefore the equality holds in \eqref{1.34} and two traces on the right-hand side coincide, i.e.,
$$
\T\,[AB-(LH)^{2m}]=\T\,[B^*B-(LH)^{2m}].
$$
This gives with the help of \eqref{1.34a}
$$
\T\big[G_{\omega_{s}}^*G_{\omega_{s}}-(LH)^{2m}\big]=\T\big[G_{\omega_{s-1}}^*G_{\omega_{s-1}}-(LH)^{2m}\big].
$$
By the induction hypothesis, the left-hand side is equal to zero. This finishes the induction step.
\end{proof}

\begin{proof}[Proof of Theorem \ref{Th1.2}.]
By Lemma \ref{L1.4}, it suffices to prove that conditions \eqref{1.24} are equivalent to \eqref{1.25}--\eqref{1.26}.

Assume first \eqref{1.24} to be valid. The second of conditions \eqref{1.24} means that the operator $L^2$ commutes with $H$. Hence $L^2$ commutes with every operator that can be written as a polynomial in $L$ and $H$. Therefore
$$
L^{\ell+1}H(LH)^{m-\ell-1}=L^2(L^{\ell-1}H(LH)^{m-\ell-1})=L^{\ell-1}H(LH)^{m-\ell-1}L^2
$$
and \eqref{1.26} holds.

Now, assume \eqref{1.25}--\eqref{1.26} to be valid. The first of conditions \eqref{1.24} holds and it remains to prove that $L^2$ commutes with $H$. We can also assume $m\geq2$. Setting $\ell=m-1$ in \eqref{1.26}, we have
\begin{equation}
L^mH=L^{m-2}HL^2
                                                    \label{1.36}
\end{equation}
and the proof is finished by applying the following statement.
\end{proof}

\begin{lemma} \label{L1.5}
Let $L$ be a self-adjoint pseudodifferential operator on $\S$. For an integer $m\ge 2$, \eqref{1.36}
is equivalent to $L^2H=HL^2$.
\end{lemma}

To prove Lemma \ref{L1.5}, we need the following known statement. Unfortunately, we do not know any reference where the statement is presented explicitly.
Therefore we present the proof too. The statement is formulated for operators on the circle although it is valid for operators on a compact manifold.

\begin{lemma} \label{L1.6}
Let $A$ and $B$ be two pseudodifferential operators on $\S$. Assume $A$ to be a self-adjoint operator.
If $A^kB=0$ for some integer $k\geq1$, then $AB=0$.
\end{lemma}

\begin{proof}
By induction in $k$, we prove the more general statement:
\begin{equation}
\mbox{if}\quad A^kBu=0\quad\mbox{for}\quad u\in C^\infty(\S),\quad\mbox{then}\quad ABu=0.
                                                   \label{1.37}
\end{equation}
The statement is trivially valid in the case of $k=1$. Assume the statement to be valid for some $k\geq1$ and for every function $u\in C^\infty(\S)$.
In particular, \eqref{1.37} holds for $B=I$, i.e.,
\begin{equation}
\mbox{if}\quad A^kv=0\quad\mbox{for}\quad v\in C^\infty(\S),\quad\mbox{then}\quad Av=0.
                                                   \label{1.38}
\end{equation}
Let now $A^{k+1}Bu=0$ for some $u\in C^\infty(\S)$. We write this in the form $A^{k}(ABu)=0$ and apply \eqref{1.38} to the function $v=ABu\in C^\infty(\S)$ to obtain $A^2Bu=0$. Then
$\|ABu\|^2=\l A^2Bu,Bu\r=0$ by the self-adjointness of $A$ which yields $ABu=0$.
\end{proof}

\begin{proof}[Proof of Lemma \ref{L1.5}.]
There is nothing to prove in the case of $m=2$. We assume $m\geq3$.

Rewrite \eqref{1.36} as
$$
L^{m-2}(L^2H-HL^2)=0.
$$
By Lemma \ref{L1.6},
\begin{equation}
L(L^2H-HL^2)=0
                                               \label{1.39}
\end{equation}
and hence
\begin{equation}
(L^2H-HL^2)L=-(L(L^2H-HL^2))^*=0.
                                               \label{1.40}
\end{equation}

Given a function $u \in C^\infty(\S)$, we consider the decomposition
$u=u_1+u_2$, where $u_1\in \overline{{\rm Ran}(L)}$ and $u_2\in \overline{{\rm Ran}(L)}^\bot={\rm Ker}(L)$ (by the self-adjointness of $L$) when we look at $L$ as an unbounded operator on $L^2(\S)$.
Then we have $(L^2H-HL^2)u_1=0$ by \eqref{1.40} and
$(L^2H-HL^2)u=L^2Hu_2$. Therefore
$$
\|(L^2H-HL^2)u\|^2=\l L^2Hu_2,(L^2H-HL^2)u\r=\l LHu_2,L(L^2H-HL^2)u\r=0
$$
by \eqref{1.39}. Hence $(L^2H-HL^2)u=0$.
\end{proof}

As the first consequence of Theorem \ref{Th1.2}, we have the following

\begin{corollary} \label{C1.1}
Let $L$ be a self-adjoint pseudodifferential operator on $\S$ such that the commutator $[L,H]$ is smoothing. If
$\T[L^{2m}-(LH)^{2m}]=0$, then $\T\big[L^{2pm}-(LH)^{2pm}\big]=0$ for every integer $p\ge1$.
\end{corollary}

\begin{proof}
We argue by induction in $p$. Assume $\T[L^{2pm}-(LH)^{2pm}]=0$ for some $p\ge 1$.
By Theorem \ref{Th1.2} applied to both $m$ and $pm$, we have
$$
L^2H=HL^2,\quad (LH)^m=(HL)^m,\quad (LH)^{pm}=(HL)^{pm}.
$$
From this
$$
(LH)^{(p+1)m}=(LH)^{pm}(LH)^m=(HL)^{pm}(HL)^m=(HL)^{(p+1)m}.
$$
Applying Theorem \ref{Th1.2} again, we derive from the latter formula
$$
\T[L^{2(p+1)m}-(LH)^{2(p+1)m}]=0.
$$
\end{proof}

We give results on optimality of Corollary \ref{C1.1}.
For $\lambda\in (0,+\infty)$, let $L_\lambda$ be the self-adjoint finite rank operator defined by
\begin{equation}
L_\lambda e^{i\theta}=e^{i\theta}+\lambda e^{-i\theta},\quad
L_\lambda e^{-i\theta}=-e^{-i\theta}+\lambda e^{i\theta},\quad
L_\lambda e^{in\theta}=0\quad\textrm{for}\quad n\notin \{-1,1\}.
                                                           \label{1.41}
\end{equation}

\begin{lemma} \label{L1.7}
The equality
\begin{equation}
L_\lambda^2H=HL_\lambda^2
                          \label{1.45}
\end{equation}
holds for every $0<\lambda\in{\mathbb R}$. For an integer $j\ge2$, the equality
\begin{equation}
(L_\lambda H)^j=(HL_\lambda)^j,
                              \label{1.46}
\end{equation}
holds if and only if
$$
\Im((1+i\lambda)^j)=0.
$$
\end{lemma}

\begin{proof}
Definition \eqref{1.41} implies that
$$
L_\lambda^2 e^{\pm i\theta} =(1+\lambda^2)e^{\pm i \theta},\quad
L_\lambda^2 e^{in\theta}=0\quad\textrm{for}\quad n\notin \{-1,1\}.
$$
Hence $L_\lambda^2$ commutes with $H$, i.e., \eqref{1.45} holds.

From \eqref{1.41} it follows that
$$
L_\lambda H e^{i\theta}=e^{i\theta}+\lambda e^{-i\theta},\quad
L_\lambda He^{-i\theta}=e^{-i\theta}-\lambda e^{i\theta},\quad
L_\lambda e^{in\theta}=0\quad\textrm{for}\quad n\notin \{-1,1\}.
$$
The plane  $P$ spanned by $\{e^{\pm i\theta}\}$ is an invariant subspace of the operator $L_\lambda H$ and the restriction of $L_\lambda H$ to $P^\bot$ is zero. The restriction $M_\lambda$ of $L_\lambda H$ to $P$ is expressed by the matrix
$$
M_\lambda=I+\lambda J,\quad
\textrm{where}\quad J=
\left(
   \begin{matrix}
            0&1\\
           -1&0
   \end{matrix}
\right)
$$
and $I$ is the unit $2\times2$-matrix.
Since $J^2=-I$,
\begin{equation}
M_\lambda^j=\sum_{\ell=0}^j{j\choose\ell}\lambda^\ell J^\ell
=\sum_{\ell=0}^{\lfloor {j/2}\rfloor}(-1)^\ell
\left(
   \begin{array}{cc}
            {j\choose {2\ell}}&{j\choose {2\ell+1}}\lambda\\[4pt]
           -{j\choose {2\ell+1}}\lambda&{j\choose {2l}}
   \end{array}
\right)\lambda^{2\ell},
                                       \label{1.48}
\end{equation}
where $\lfloor x\rfloor$ is the integer part of $x\geq0$ and the following agreement is used: ${j\choose {2\ell+1}}=0$ if $j<2\ell+1$.
Equality \eqref{1.46} is equivalent to
$$
M_\lambda^j=\left(
   \begin{matrix}
            1&0\\
           0&-1
   \end{matrix}
\right)M_\lambda^j\left(
   \begin{matrix}
            1&0\\
           0&-1
   \end{matrix}
\right).
$$
On using \eqref{1.48}, we see that \eqref{1.46} is equivalent to
$$
\sum_{\ell=0}^{\lfloor {j/2}\rfloor}(-1)^\ell{j\choose {2\ell+1}}\lambda^{2\ell+1}=0.
$$
Finally, we observe that the left-hand side of the last equality is $\Im((1+i\lambda)^j)$.
\end{proof}

The following statement demonstrates the optimality of Corollary \ref{C1.1}.

\begin{proposition}  \label{P1.1}
Given an integer $m\ge 2$, set $\lambda=\tan({\pi\over m})$. For an integer $j\ge 1$, the equality
$$
\T[L_\lambda^{2j}-(L_\lambda H)^{2j}]=0
$$
holds if and only if $j$ is an integer multiple of $m$.
\end{proposition}

\begin{proof}
As is seen from
$$
\Im((1+i\lambda)^j)=\big(\cos{\pi\over m}\big)^{-j}\Im\big(e^{i{{j\pi}\over m}}\big),
$$
\eqref{1.46} holds if and only if
$j$ is an integer multiple of $m$.
Applying Lemma \ref{L1.7}, we obtain the statement:
$(L_\lambda H)^j=(HL_\lambda)^j$ if and only if $j$ is an integer multiple of $m$.
By Theorem \ref{Th1.2} and \eqref{1.45}, the latter statement is equivalent to the proposition.
\end{proof}

The following result sharpens  Corollary \ref{C1.1}.

\begin{theorem} \label{Th1.3}
Let $L$ be a self-adjoint pseudodifferential operator on $\S$ such that the commutator $[L,H]$ is smoothing.
Let $m$ be the greatest common factor of integers $m_1>1$ and $m_2>1$.
If
$$
\T[L^{2m_1}-(LH)^{2m_1}]=\T[L^{2m_2}-(LH)^{2m_2}]=0,
$$
then  $(LH)^m=(HL)^m$ and $\T[L^{2pm}-(LH)^{2pm}]=0$ for any integer $p\ge 1$.
\end{theorem}

Together with inequality \eqref{1.9}, this implies

\begin{corollary} \label{C1.2}
Let $L$ be a self-adjoint pseudodifferential operator on $\S$ such that the commutator $[L,H]$ is smoothing.
Either $\T[L^{2p}-(LH)^{2p}]>0$ for every $p\in \N$ or there exists an integer $m\ge1$ such that, for every $k\in\N$,
$\T[L^{2k}-(LH)^{2k}]=0$ if and only if $k$ is an integer multiple of $m.$
\end{corollary}

\begin{proof}[Proof of Theorem \ref{Th1.3}]
If $m_2$ is an integer multiple of $m_1$ or $m_1$ is an integer multiple of $m_2$, the statement follows from Corollary \ref{C1.1}.
For the rest of the proof we assume that $m<\min(m_1,m_2)$.
Choose integers $a$ and $b$  such that $am_1+bm_2=m$. We can assume without loss of generality that $a>0$ and $b<0$ (otherwise invert roles of $m_1$ and $m_2$).

By \eqref{1.24}, $L^2$ commutes with $H$ and the following equalities hold:
$$
(LH)^{m_1}=(HL)^{m_1},\quad (LH)^{m_2}=(HL)^{m_2}.
$$
On using these equalities and $am_1=-bm_2+m$, we derive
$$
\begin{aligned}
(LH)^{am_1}&=(LH)^{-bm_2+m}=(LH)^{-bm_2}(LH)^m,\\
(LH)^{am_1}&=(HL)^{am_1}=(HL)^{-bm_2+m}=(HL)^{-bm_2}(HL)^m=(LH)^{-bm_2}(HL)^m.
\end{aligned}
$$
This implies
$$
(LH)^{-bm_2}\big((LH)^m-(HL)^m\big)=0.
$$
On using Lemma \ref{L1.8} that is presented below, we obtain $(LH)^m=(HL)^m$. Then the statement of the theorem follows from Corollary \ref{C1.1}.
\end{proof}

\begin{lemma}   \label{L1.8}
Let $L$ be a self-adjoint pseudodifferential operator on $\S$ such that $L^2$ commutes with $H$.
If $(LH)^k\big((LH)^m-(HL)^m\big)=0$ for some positive integers $k$ and $m$, then $(LH)^m=(HL)^m$.
\end{lemma}

\begin{proof}
We first prove by induction in $\ell$ the following statement: for every $u\in C^\infty(\S)$,
\begin{equation}
\textrm{if}\quad(LH)^\ell Lu=0\quad\textrm{for some}\quad \ell\in \N,\quad\textrm{then}\quad Lu=0.
                                              \label{1.52}
\end{equation}
There is nothing to prove when $l=0$.
Assume \eqref{1.52} to be valid for some $l\ge0$ and assume that
$(LH)^{\ell+1}Lv=(LH)^\ell(LHLv)=0$ for some $v\in C^\infty(\S)$. Applying the induction hypothesis to the function $u=HLv$, we obtain
$$
LHLv=0.
$$
On using the self-adjointness of $L$ and permutability of $L^2$ and $H$, we derive from the last formula
$$
\begin{aligned}
0&=\|LHLv\|^2=\l LHLv,LHLv\r=\l L^2HLv,HLv\r\\&=\l HL^3v,HLv\r
=\l L^3v,Lv\r=\|L^2v\|^2.
\end{aligned}
$$
We have also used that $H$ is a unitary operator. Hence $Lv=0$. We have thus proved \eqref{1.52}.

Now, the equalities
$$
\begin{aligned}
0=(LH)^k((LH)^m-(HL)^m)&=(LH)^{k-1}L H((LH)^m-(HL)^m)\\&=(LH)^{k-1}L ((HL)^m-(LH)^m)H
\end{aligned}
$$
imply with the help of \eqref{1.52} that
\begin{equation}
L ((LH)^m-(HL)^m)=0.
                                        \label{1.54}
\end{equation}

Given $u\in C^\infty(\S)$, we have by \eqref{1.54}
\begin{equation}
\begin{aligned}
\|((L&H)^m-(HL)^m)u\|^2\\&=\l L((LH)^m-(HL)^m)u,(HL)^{m-1}Hu\r
-\l ((LH)^m-(HL)^m)u,(HL)^{m}u\r\\
&=-\l ((LH)^m-(HL)^m)u,(HL)^{m}u\r.
\end{aligned}
                                            \label{1.55}
\end{equation}
Represent $u$ as the sum $u=u_1+u_2$, where $u_1\in \overline{{\rm Ran}(L)}$ and $u_2\in \overline{{\rm Ran}(L)}^\bot={\rm Ker}(L)$ when we look at $L$ as an unbounded operator in $L^2(\S)$.
Then $(HL)^{m}u=(HL)^{m}u_1$ and \eqref{1.55} gives
$$
\|((LH)^m-(HL)^m)u\|^2=-\l (LH)^m ((LH)^m-(HL)^m)u,u_1\r.
$$
Choose a sequence $u_{1,n}\in C^\infty(\S)\ (n=1,2,\dots)$ such that $Lu_{1,n}$ converges to $u_1$ in $L^2(\S)$ as $n\to +\infty$. Then
$$
\begin{aligned}
\|((LH)^m-(HL)^m)u\|^2&=-\lim_{n\to\infty}\l (LH)^m ((LH)^m-(HL)^m)u,Lu_{1,n}\r\\
&=-\lim_{n\to\infty}\l L(LH)^m ((LH)^m-(HL)^m)u,u_{1,n}\r\\
&=-\lim_{n\to\infty}\l L^2(HL)^{m-1}H ((LH)^m-(HL)^m)u,u_{1,n}\r.
\end{aligned}
$$
On using the permutability of $L^2$ and $H$, we obtain
$$
\|((LH)^m-(HL)^m)u\|^2=-\lim_{n\to\infty}\l (HL)^{m-1}H L^2((LH)^m-(HL)^m)u,u_{1,n}\r.
$$
By \eqref{1.54}, the right-hand side of the latter formula equals zero and we obtain
$$
((LH)^m-(HL)^m)u=0.
$$
\end{proof}

\section{Proof of Theorem \ref{Th0.1}}

Let $\Lambda^{1/2}:C^\infty({\mathbb S})\rightarrow C^\infty({\mathbb S})$ be the nonnegative self-adjoint operator  satisfying $(\Lambda^{1/2})^2=\Lambda$. In other words, $\Lambda^{1/2}$ is defined by
\begin{equation}
\Lambda^{1/2}e^{in\theta}=\sqrt{|n|}e^{in\theta}\quad (n\in\Z).
                                   \label{2.1}
\end{equation}
The operator $D$ can be expressed through $\Lambda^{1/2}$ and $H$:
\begin{equation}
D=\Lambda^{1/2}H\Lambda^{1/2}.
                                   \label{2.2}
\end{equation}
Obviously, operators $\Lambda^{1/2}$ and $H$ commute.

Given a real function $a\in C^\infty(\S)$,
define the operator $L:C^\infty({\mathbb S})\rightarrow C^\infty({\mathbb S})$ by formula \eqref{0.6},
where $a$ stands for the operator of multiplication by the function $a$. We do not designate the dependence of $L$ on $a$ explicitly in order to use formulas of the previous section without changing notations. Nevertheless, the reader should remember that $L$ depends on $a$.

Observe that $L$ is a self-adjoint first order pseudodifferential operator and satisfies
\begin{equation}
[L,H]=\Lambda^{1/2}[a,H]\Lambda^{1/2}.
                                   \label{2.3}
\end{equation}
The commutator $[a,H]$ is a smoothing operator for every function $a\in C^\infty(\S)$ \cite[Section 5.4]{JS}. We see from \eqref{2.3} that $[L,H]$ is also a smoothing operator.
By Theorem \ref{Th1.1},
\begin{equation}
\T[L^{2m}-(LH)^{2m}]\ge 0.
                                   \label{2.4}
\end{equation}

As follows from \eqref{0.6} and \eqref{2.2},
\begin{equation}
L^{2m}=\Lambda^{1/2}(a\Lambda)^{2m-1}a\Lambda^{1/2},\quad
(LH)^{2m}=\Lambda^{1/2}(aD)^{2m-1}a\Lambda^{1/2}H.
                                   \label{2.4a}
\end{equation}
Substitute these values into \eqref{2.4} to obtain
$$
\T\big[\Lambda^{1/2}(a\Lambda)^{2m-1}a\Lambda^{1/2}-\Lambda^{1/2}(aD)^{2m-1}a\Lambda^{1/2}H\big]\ge 0.
$$
With the help of the trigonometric basis, this can be written as
$$
\sum\limits_{n\in{\mathbb Z}}\Big(\big\langle(a\Lambda)^{2m-1}a\Lambda^{1/2}e^{in\theta},\Lambda^{1/2}e^{in\theta}\big\rangle
-\big\langle(aD)^{2m-1}a\Lambda^{1/2}He^{in\theta},\Lambda^{1/2}e^{in\theta}\big\rangle\Big)\ge 0.
$$

As is seen from \eqref{1.0} and \eqref{2.1},
$$
\big\langle(a\Lambda)^{2m-1}a\Lambda^{1/2}e^{in\theta},\Lambda^{1/2}e^{in\theta}\big\rangle
=|n|\big\langle(a\Lambda)^{2m-1}a e^{in\theta},e^{in\theta}\big\rangle
=\big\langle(a\Lambda)^{2m} e^{in\theta},e^{in\theta}\big\rangle,
$$
$$
\big\langle(aD)^{2m-1}a\Lambda^{1/2}He^{in\theta},\Lambda^{1/2}e^{in\theta}\big\rangle
=n\big\langle(aD)^{2m-1}a e^{in\theta},e^{in\theta}\big\rangle
=\big\langle(aD)^{2m} e^{in\theta},e^{in\theta}\big\rangle.
$$
Therefore
$$
\begin{aligned}
&\T[L^{2m}-(LH)^{2m}]\\
&=\sum\limits_{n\in{\mathbb Z}}\Big(\big\langle(a\Lambda)^{2m-1}a\Lambda^{1/2}e^{in\theta},\Lambda^{1/2}e^{in\theta}\big\rangle
-\big\langle(aD)^{2m-1}a\Lambda^{1/2}He^{in\theta},\Lambda^{1/2}e^{in\theta}\big\rangle\Big)\\
&=\sum\limits_{n\in{\mathbb Z}}\Big(\big\langle(a\Lambda)^{2m}e^{in\theta},e^{in\theta}\big\rangle
-\big\langle(aD)^{2m}e^{in\theta},e^{in\theta}\big\rangle\Big)
=\T[(a\Lambda)^{2m}-(aD)^{2m}]=Z_m(a).
\end{aligned}
$$
In particular, we have proved \eqref{0.7}. Together with \eqref{0.7}, inequality \eqref{2.4} proves the second statement of Theorem \ref{Th0.1}: $Z_m(a)\ge0$.

\begin{lemma} \label{L2.1}
Given a real function $a\in C^\infty(\S)$ and integer $m\ge 1$, define functions $g_n\in C^\infty({\mathbb S})\ (n\in{\mathbb Z})$ by
\begin{equation}
g_n(\theta)=(aD)^{m-1}ae^{in\theta}.
                                   \label{2.7}
\end{equation}
Then
\begin{equation}
Z_m(a)\ge 4\sum\limits_{n>0,k>0}nk|(\widehat{g_n})_{-k}|^2.
                                   \label{2.8}
\end{equation}
\end{lemma}

\begin{proof}
By Theorem \ref{Th1.1},
$$
\T\,[L^{2m}-(LH)^{2m}]\ge \T\,\big[H(LH)^{m-1}L^2H(LH)^{m-1}-(LH)^{2m}\big].
$$
By \eqref{0.7}, the left-hand side coincides with $Z_m(a)$, i.e., the inequality can be written as
\begin{equation}
Z_m(a)\ge \T\,\big[H(LH)^{m-1}L^2H(LH)^{m-1}-(LH)^{2m}\big].
                                   \label{2.9}
\end{equation}

As follows from \eqref{2.2} and \eqref{0.6},
$$
H(LH)^{m-1}L^2H(LH)^{m-1}=H\Lambda^{1/2}(aD)^{m-1}a\Lambda (aD)^{m-1}a\Lambda^{1/2}H.
$$
Substitute this value and \eqref{2.4a} into \eqref{2.9}
$$
Z_m(a)\ge \T\,\big[H\Lambda^{1/2}(aD)^{m-1}a\Lambda (aD)^{m-1}a\Lambda^{1/2}H-\Lambda^{1/2}(aD)^{2m-1}a\Lambda^{1/2}H\big].
$$
On using the trigonometric basis, we write this in the form
\begin{equation}
\begin{aligned}
Z_m(a)\ge \sum\limits_{n\in{\mathbb Z}}\Big(&\big\langle H\Lambda^{1/2}(aD)^{m-1}a\Lambda (aD)^{m-1}a\Lambda^{1/2}He^{in\theta},e^{in\theta}\big\rangle\\
&-\big\langle\Lambda^{1/2}(aD)^{2m-1}a\Lambda^{1/2}He^{in\theta},e^{in\theta}\big\rangle\Big).
\end{aligned}
                                   \label{2.10}
\end{equation}

On using the equalities $He^{in\theta}=\mbox{sgn}(n)e^{in\theta}$ and $\Lambda^{1/2}e^{in\theta}=|n|^{1/2}e^{in\theta}$, we transform the first summand on the right-hand side of \eqref{2.10} as follows:
$$
\begin{aligned}
\big\langle H\Lambda^{1/2}(aD)^{m-1}a\Lambda (aD)^{m-1}&a\Lambda^{1/2}He^{in\theta},e^{in\theta}\big\rangle\\
&=\big\langle (aD)^{m-1}a\Lambda (aD)^{m-1}a\Lambda^{1/2}He^{in\theta},\Lambda^{1/2}He^{in\theta}\big\rangle\\
&=|n|\big\langle (aD)^{m-1}a\Lambda (aD)^{m-1}ae^{in\theta},e^{in\theta}\big\rangle\\
&=|n|\big\langle \big((aD)^{m-1}a\big)^*\Lambda (aD)^{m-1}ae^{in\theta},e^{in\theta}\big\rangle\\
&=|n|\big\langle \Lambda (aD)^{m-1}ae^{in\theta},(aD)^{m-1}ae^{in\theta}\big\rangle
=|n|\langle \Lambda g_n,g_n\rangle.
\end{aligned}
$$
We have used \eqref{2.7} for the last equality of the chain.
The second summand on the right-hand side of \eqref{2.10} is transformed similarly:
$$
\begin{aligned}
\big\langle \Lambda^{1/2}(aD)^{2m-1}a\Lambda^{1/2}&He^{in\theta},e^{in\theta}\big\rangle
=\big\langle (aD)^{2m-1}a\Lambda^{1/2}He^{in\theta},\Lambda^{1/2}e^{in\theta}\big\rangle\\
&=n\big\langle (aD)^{2m-1}ae^{in\theta},e^{in\theta}\big\rangle
=n\big\langle \big((aD)^{m-1}a\big)^*D (aD)^{m-1}ae^{in\theta},e^{in\theta}\big\rangle\\
&=n\big\langle D (aD)^{m-1}ae^{in\theta},(aD)^{m-1}ae^{in\theta}\big\rangle
=n\langle D g_n,g_n\rangle.
\end{aligned}
$$
Now, \eqref{2.10} takes the form
\begin{equation}
Z_m(a)\ge \sum\limits_{n\in{\mathbb Z}}\Big(|n|\langle \Lambda g_n,g_n\rangle
-n\langle D g_n,g_n\rangle\Big).
                                   \label{2.11}
\end{equation}

We transform the sum on the right-hand side of (\ref{2.11}) as follows:
\begin{equation}
\begin{aligned}
\sum_{n\in \Z}\big(&|n|\langle \Lambda g_n,g_n\rangle-n\langle Dg_n,g_n\rangle\big)=\\
&=\sum_{n>0}n\big(\langle \Lambda g_n,g_n\rangle-\langle Dg_n,g_n\rangle\big)
+\sum_{n<0}(-n)\big(\langle \Lambda g_n,g_n\rangle+\langle Dg_n,g_n\rangle\big)\\
&=\sum_{n>0}n\langle (\Lambda-D) g_n,g_n\rangle
+\sum_{n>0}n\langle (\Lambda+D) g_{-n},g_{-n}\rangle.
\end{aligned}
                                 \label{2.12}
\end{equation}
The operator $D$ satisfies $\overline{Dg}=-D\bar g$. Therefore
$$
\overline{g_n}=\overline{(aD)^{2m-1} a e^{in\theta}}=-(aD)^{2m-1}a e^{-in\theta}
=-g_{-n}.
$$
Formula (\ref{2.12}) is now written as
\begin{equation}
\sum_{n\in \Z}\big(|n|\langle \Lambda g_n,g_n\rangle-n\langle Dg_n,g_n\rangle\big)
=\sum_{n>0}n\langle (\Lambda-D) g_n,g_n\rangle
+\sum_{n>0}n\langle (\Lambda+D) \overline{g_{n}},\overline{g_{n}}\rangle.
                                 \label{2.13}
\end{equation}

For any function $g$,
$$
\langle (\Lambda-D) g,g\rangle=2\sum\limits_{k>0}k|{\hat g}_{-k}|^2,\quad
\langle (\Lambda+D) g,g\rangle=2\sum\limits_{k>0}k|{\hat g}_{k}|^2.
$$
Therefore (\ref{2.13}) takes the form
$$
\sum_{n\in \Z}\big(|n|\langle \Lambda g_n,g_n\rangle-n\langle Dg_n,g_n\rangle\big)
=2\sum_{n>0,k>0}nk|(g_n)^\wedge_{-k}|^2
+2\sum_{n>0,k>0}nk|(\overline{g_n})^\wedge_{k}|^2.
$$
Since $(\bar g)^\wedge_k=\overline{{\hat g}_{-k}}$ for every function $g$, the latter formula can be written as
$$
\sum_{n\in \Z}\big(|n|\langle \Lambda g_n,g_n\rangle-n\langle Dg_n,g_n\rangle\big)
=4\sum_{n>0,k>0}nk|(\widehat{g_n})_{-k}|^2.
$$
Replacing the right-hand side of (\ref{2.11}) with the latter expression, we arrive to (\ref{2.8}).
\end{proof}

\begin{proof}[Proof of Theorem \ref{Th0.1}]
For $m=1$, estimate (\ref{0.4}) follows from Edward's formula (\ref{0.3}).
Therefore we assume $m\geq2$. Such $m$ is fixed till the end of the proof as well as a real function $a\in C^{\infty}(\S)$. The dependence of different quantities on $m$ and $a$ is not designated explicitly.

Obviously, for $1\leq n\in \N$,
\begin{equation}
(aD)^{m-1} a e^{in \theta}=\Big(\sum_{s=1}^m n^{s-1} f_s \Big)e^{in\theta},
                   \label{3.2}
\end{equation}
where
\begin{equation}
f_{m}=a^m=b
                   \label{3.3}
\end{equation}
and
$$
f_s=P_s(a, Da,\dots, D^{m-s}a)\quad\textrm{for}\quad s=1,\dots, m-1
$$
with some universal polynomials $P_s(X_1,\ldots, X_{m-s})$ in $m-s$ variables.

Define functions $g_n\in C^\infty({\mathbb S})\ (n\in{\mathbb Z})$ by (\ref{2.7}). With the help of (\ref{3.2}), we see that
$$
(\widehat{g_n})_{-k}=\sum_{s=1}^m n^{s-1}(\widehat{f_s})_{-(n+k)}\quad (k=1,2,\dots).
$$
Substituting this value into (\ref{2.8}), we obtain
\begin{equation}
Z_m(a)\geq
4\sum_{n,k=1}^\infty nk\,\Big|\sum_{s=1}^m n^{s-1}(\widehat{f_s})_{-(n+k)}\Big|^2
=4\sum_{j=2}^\infty\sum_{n=1}^{j-1}n(j-n)\,\Big|\sum_{s=1}^m n^{s-1}(\widehat{f_s})_{-j}\Big|^2.
                   \label{3.4}
\end{equation}

For every integer $j\geq2$, introduce the sesquilinear form $F_j:{\mathbb C}^m\times{\mathbb C}^m\rightarrow{\mathbb C}$ by
$$
F_j(x,y)={1\over j}\sum_{n=1}^{j-1}{n\over j}(1-{n/j})\sum_{s=1}^m (n/j)^{s-1}x_s
\sum_{t=1}^m (n/j)^{t-1}\overline{y_t}.
$$
Observe that the corresponding Hermitian form is non-negative:
\begin{equation}
F_j(x,x)={1\over j}\sum_{n=1}^{j-1}{n\over j}(1-n/j)\Big|\sum_{s=1}^m (n/j)^{s-1}x_s\Big|^2\geq0.
                                                            \label{3.5}
\end{equation}

For $j\geq2$, define the vector ${\tilde f}_j\in{\mathbb C}^m$ by
\begin{equation}
({\tilde f}_j)_s=j^{s-1}(\widehat{f_s})_{-j}\quad(1\leq s\leq m).
                                                            \label{3.6}
\end{equation}
Then
$$
\sum_{n=1}^{j-1}n(j-n)\,\Big|\sum_{s=1}^m n^{s-1}(\widehat{f_s})_{-j}\Big|^2=j^3F_j({\tilde f}_j,{\tilde f}_j)
$$
and inequality (\ref{3.4}) can be written in the form
$$
Z_m(a)\geq 4\sum_{j=2}^\infty j^3F_j({\tilde f}_j,{\tilde f}_j).
$$
Since all summands on the right-hand side are non-negative, this implies
\begin{equation}
Z_m(a)\geq 4\sum_{j=m+1}^\infty j^3F_j({\tilde f}_j,{\tilde f}_j).
                                                            \label{3.7}
\end{equation}

\begin{lemma} \label{L3.1}
There exists a positive constant $c_m$ depending only on $m$ such that
$$
F_j(x,x)\geq c_m\|x\|^2
$$
for every $j\geq m+1$ and for every $x\in{\mathbb C}^m$, where $\|x\|=\big(\sum_{s=1}^m|x_s|^2\big)^{1/2}$ is the standard norm on ${\mathbb C}^m$.
\end{lemma}

The proof of the lemma is presented at the end of the section, and now we finish the proof of Theorem \ref{Th0.1}.

With the help of Lemma \ref{L3.1}, we derive from (\ref{3.7})
$$
Z_m(a)\geq 4c_m\sum_{j=m+1}^\infty j^3\|{\tilde f}_j\|^2\geq 4c_m\sum_{j=m+1}^\infty j^3\big|({\tilde f}_j)_m\big|^2.
$$
By (\ref{3.3}) and (\ref{3.6}), $({\tilde f}_j)_m=j^{m-1}{\hat b}_{-j}$. Substituting this value into the last inequality, we obtain
$$
Z_m(a)\geq 4c_m\sum_{j=m+1}^\infty j^{2m+1}|{\hat b}_{-j}|^2.
$$
This coincides with desired inequality (\ref{0.4}) because $|{\hat b}_{-j}|=|{\hat b}_{j}|$.
\end{proof}

\begin{proof}[Proof of Lemma \ref{L3.1}.]
Introduce the sesquilinear form $F_\infty:{\mathbb C}^m\times{\mathbb C}^m\rightarrow{\mathbb C}$ by
$$
F_\infty(x,y)=\int_0^1 t(1-t)\sum_{s=0}^{m-1}t^{s-1}x_s\sum_{s'=0}^{m-1}t^{s'-1}\overline{y_{s'}}\,dt.
$$
The right-hand side of (\ref{3.5}) is the Riemann integral sum for the integral
\begin{equation}
F_\infty(x,x)=\int_0^1 t(1-t)\Big|\sum_{s=1}^mt^{s-1}x_s\Big|^2\,dt.
                                                     \label{3.8}
\end{equation}
Therefore $F_\infty=\lim_{j\rightarrow\infty}F_j$. Since the Hermitian forms $F_j(x,x)$ and $F_\infty(x,x)$ are non-negative, it suffices to prove that, for $j\geq m+1$,
$$
F_\infty(x,x)\neq 0\quad\mbox{and}\quad F_j(x,x)\neq 0\quad\mbox{for}\quad 0\neq x\in{\mathbb C}^m.
$$

In view of (\ref{3.8}), the equality $F_\infty(x,x)=0$ means that
$$
\sum_{s=1}^mt^{s-1}x_s=0\quad\mbox{for}\quad t\in(0,1).
$$
Choosing a sequence $0<t_1<\dots<t_m<1$, we obtain the linear system
$$
\sum_{s=1}^mt_k^{s-1}x_s=0\quad(1\leq k\leq m)
$$
with the non-degenerate matrix $(t_k^{s-1})_{1\leq k,s\leq m}$. Therefore $x=0$.

Let $j\geq m+1$. In view of (\ref{3.5}), the equality $F_j(x,x)=0$ implies
$$
\sum_{s=1}^m (n/j)^{s-1}x_s=0\quad(1\leq n\leq m).
$$
This is again a linear system with non-degenerate matrix which implies $x=0$.
\end{proof}

\section{The null space of zeta-invariants}

In the case of a general operator $L$, Corollary \ref{C1.1} cannot be improved.
However, Corollary \ref{C1.1} is greatly improved by Theorem \ref{Th0.3} when $L=\Lambda^{1/2}a\Lambda^{1/2}$ for a real function $a\in C^\infty(\S)$.

\begin{proof}[Proof of Theorem \ref{Th0.3}.]
The case of $m=1$ follows from Edward's formula \eqref{0.3}.
We assume $m\ge 2$ for the rest of the proof.

We first prove the ``if'' statement.
Assume \eqref{0.5} to be valid.
Then
\begin{equation}
a\Lambda=aDH=HaD,
                                    \label{4.2}
\end{equation}
i.e., the operators $aD$ and $H$ commute.
Indeed, for $n\in \Z$,
$$
\begin{aligned}
HaDe^{in\theta}&=nHae^{in\theta}=n\hat a_0 H(e^{in\theta})+n\hat a_1 H(e^{i(n+1)\theta})+n\hat a_{-1}H( e^{i(n-1)\theta})\\
&=|n|\hat a_0 e^{in\theta}+|n|\hat a_1 e^{i(n+1)\theta}+|n|\hat a_{-1}e^{i(n-1)\theta}
=a\Lambda e^{in\theta}.
\end{aligned}
$$
As is seen from \eqref{4.2} and the identity $H^2=I$, $(a\Lambda)^{2m}=(aD)^{2m}$ and hence $Z_m(a)=0$.

Let $\stackrel\circ{\D}\,\,=\{z\in\C\mid|z|<1\}$ be the interior of the unit disk $\D$. We say that a function
$\varphi\in C(\S)$ admits a holomorphic extension to $\D$ if there exists $\Phi\in C(\D)$ which is holomorphic in $\stackrel\circ{\D}$ and such that
$\Phi|_{\S}=\varphi$.

Now, we prove the ``only if" part.
Assume that $Z_m(a)=0$. By \eqref{0.7},
$$
\T[L^m-(LH)^m]=Z_m(a)=0
$$
for the operator $L=\Lambda^{1/2}a\Lambda^{1/2}$. Apply Theorem \ref{Th1.2} to obtain $L^2H=HL^2$, i.e.,
$$
\Lambda^{1/2}a\Lambda a\Lambda^{1/2}H
=H\Lambda^{1/2}a\Lambda a\Lambda^{1/2}.
$$
This implies
$$
\Lambda^{1/2}a\Lambda a\Lambda^{1/2}He^{i\theta}=H\Lambda^{1/2}a\Lambda a\Lambda^{1/2}e^{i\theta}.
$$
Since $\Lambda^{1/2}He^{i\theta}=\Lambda^{1/2}e^{i\theta}=e^{i\theta}$, the last formula can be written as
\begin{equation}
\Lambda^{1/2}\tilde h=H\Lambda^{1/2}\tilde h
                                   \label{4.4}
\end{equation}
for the function $\tilde h$ defined by
\begin{equation}
\tilde h=a\Lambda ae^{i\theta}.
                                  \label{4.5}
\end{equation}
Applying the operator $\Lambda^{1/2}$ to equality \eqref{4.4}, we obtain $(\Lambda-\Lambda^{1/2}H\Lambda^{1/2})\tilde h=0$. With the help of \eqref{2.2},
this can be written as
\begin{equation}
(\Lambda-D)\tilde h=0.
                                   \label{4.6}
\end{equation}
As follows from definitions of $\Lambda$ and $D$, $(\Lambda-D)\tilde h=-2\sum_{k<0}k{\hat {\tilde h}}_ke^{ik\theta}$. Therefore each of equations \eqref{4.4} and \eqref{4.6} is equivalent to the statement:
\begin{equation}
\tilde h\ \textrm{admits a holomorphic extension to}\ \D.
                                    \label{4.7}
\end{equation}

Next, we prove that
\begin{equation}
\delta=(ae^{i\theta})^m \textrm{ admits a holomorphic extension to }\D.
                                      \label{4.8}
\end{equation}
This follows from \eqref{0.4}. Indeed, since $Z_m(a)=0$, we see that $\widehat{(a^m)}_j=0$ for $|j|\ge m+1$. Hence $\delta$
admits the holomorphic extension given by the polynomial
$$
\sum_{j=0}^{2m}\widehat{(a^m)}_{j-m}z^j\quad (z\in \D).
$$

We assume for the rest of the proof that $a$ is not identically zero. In particular the function $\delta$ in \eqref{4.8} is not identically zero  and  $a(\theta)\not=0$ for $\theta$ belonging to a dense subset of $\R$.
Then from \eqref{4.5} it follows that
$$
\tilde h=a\Lambda a e^{i\theta}=
a(\Lambda+D)(a e^{i\theta})-(aDa) e^{i\theta}-a^2e^{i\theta}.
$$
Multiply this equation by $\delta$
\begin{equation}
\delta\tilde h=a\delta(\Lambda+D)(a e^{i\theta})-\delta(aDa) e^{i\theta}-\delta a^2e^{i\theta}.
                                                \label{4.10}
\end{equation}
As easily follows from definition \eqref{4.8} of $\delta$,
$$
\delta(aDa) e^{i\theta}+\delta a^2e^{i\theta}=\frac{1}{m}a^2(D\delta)e^{i\theta}.
$$
Therefore \eqref{4.10} can be written as
\begin{equation}
\delta\tilde h=a\delta (\Lambda+D)(a e^{i\theta})-{1\over m}a^2e^{i\theta}D\delta.
                                                     \label{4.12}
\end{equation}

Introduce the functions
\begin{equation}
f= -{1\over m}e^{i\theta}D\delta,\quad
g=\delta(\Lambda+D)(a e^{i\theta}),\quad
h=\delta\tilde h.
                                                     \label{4.14}
\end{equation}
Then equation \eqref{4.12} is written as
\begin{equation}
h=a^2f+ag.
                                                     \label{4.15}
\end{equation}
Multiply this equation by $a^{m-2}$
$$
a^m f=a^{m-2}h-a^{m-1}g.
$$
As is seen from definition \eqref{4.8} of $\delta$, $a^m=e^{-im\theta}\delta$. Substitute this expression for $a$ into the left-hand side of the last formula
\begin{equation}
e^{-mi\theta}\delta f=a^{m-2}h-a^{m-1}g.
                                                     \label{4.16}
\end{equation}
Multiplying this equation by $a$ and using $a^m=e^{-im\theta}\delta$ again, we obtain one more equation
\begin{equation}
e^{-mi\theta}\delta g=a^{m-1}h-ae^{-mi\theta}\delta f.
                                                     \label{4.16a}
\end{equation}

The function $(\Lambda+D)u$ admits a holomorphic extension to $\D$ for every $u\in C^\infty(\S)$. Besides this, if functions $u$ and $v$ admit holomorphic extensions to $\D$, then the functions $u+v,\ uv$, and $Du$ are also holomorphically extendible to $\D$.
Therefore each of the functions $f,g,h$ admits a holomorphic extension to $\D$ as is seen from \eqref{4.7}--\eqref{4.8} and \eqref{4.14}.
Since $a$ is a real function that does not vanish on a dense subset of the circle,
$$
(\Lambda+D)(a e^{i\theta}),\ a\Lambda ae^{i\theta},\textrm{ and }  D\delta\textrm{ are not identically zero}.
$$
Hence $f$, $g$, and $h$ are not identically zero.

We will use the following

\begin{lemma}    \label{L4.1}
Let a real function $b\in C^\infty(\S)$ be such that $b^se^{is\theta}$ admits a holomorphic extension to $\D$ for some integer $s\ge1$.
If $b$ satisfies the equation
\begin{equation}
b^r\varphi^r=\psi^r
                                                         \label{4.18}
\end{equation}
for some integer $r\ge1$, where both functions $\varphi,\psi\in C(\S)$ admit holomorphic extensions to $\D$ and $\psi$ is not identically zero, then the function $b$
is of the form
\begin{equation}
b(\theta)={\hat b}_0+2\Re ({\hat b}_1 e^{i\theta})\quad
\mbox{for some }{\hat b}_0\in\R\mbox{ and }{\hat b}_1\in\C.
                                             \label{4.19}
\end{equation}
\end{lemma}
The proof of the lemma is given at the end of the section.

Now, we consider separately the case of $m=2$. In this case $\delta=e^{2i\theta}a^2$. Multiply equation \eqref{4.15} by $e^{2i\theta}$
$$
a e^{2i\theta} g=-\delta f+e^{2i\theta}h.
$$
Hypotheses of Lemma \ref{L4.1} are satisfied for $(b,\varphi,\psi,r,s)=(a,e^{2i\theta} g,-\delta f+e^{2i\theta}h,1,m)$.
Applying Lemma \ref{L4.1}, we obtain that $a$ is of form \eqref{0.5}. We assume $m\ge3$ for the rest of the proof.

Let us also consider separately the case of $m=3$.  We write system \eqref{4.15}--\eqref{4.16a} in the matrix form
\begin{equation}
\left(
\begin{array}{cc}
-g&-f\\
h&-g\\
-e^{-3i\theta}\delta f&h
\end{array}
\right)
\left(
\begin{array}{c}
a\\
a^2
\end{array}
\right)
=
\left(
\begin{array}{c}
-h\\
e^{-3i\theta}\delta f\\
e^{-3i\theta}\delta g
\end{array}
\right).
                                             \label{4.21}
\end{equation}

First, assume that $w_2=g^2+fh$ is not identically zero. Eliminating $a^2$ from first two equations of system \eqref{4.21}, we obtain
$$
a(e^{3i\theta}w_2)= \delta f^2+e^{3i\theta}gh.
$$
Hypotheses of Lemma \ref{L4.1} are satisfied for $(b,\varphi,\psi,r,s)=(a,e^{3i\theta} w_2,\delta f^2+e^{3i\theta}gh,1,m)$.
Applying Lemma \ref{L4.1}, we obtain \eqref{0.5}.

Next, assume that $w_2$ is identically zero but $\beta=\delta f^2+e^{i3\theta}gh$ is not identically zero.
Eliminating $a^2$ from first and third equations of system \eqref{4.21}, we obtain
$$
a\beta=-\delta fg+e^{3i\theta}h^2.
$$
This again implies \eqref{0.5} with the help of Lemma \ref{L4.1}.

Finally, consider the case when both $\beta$  and $w_2$ are identically zero, i.e., when
$$
g^2+fh=0,\quad a^3f^2+gh=0.
$$
This implies
$$
a^3f^3=g^3.
$$
Applying Lemma \ref{L4.1} with $(b,\varphi,\psi,r,s)=(a,f,g,3,m)$, we obtain \eqref{0.5}.

We assume $m\ge4$ for the rest of the proof.
Multiply equation \eqref{4.15} by $a^\ell$
\begin{equation}
a^\ell h=a^{\ell+2}f+a^{\ell+1}g\quad(0\le\ell\le m-1).
                                            \label{4.22}
\end{equation}
We combine equations \eqref{4.22} and \eqref{4.14}--\eqref{4.15} into the system
$$
\begin{aligned}
-g&a-fa^2=-h,\\
h&a-ga^2-fa^3=0,\\
h&a^2-ga^3-fa^4=0,\\
&\dots\dots\dots\\
h&a^{m-3}-ga^{m-2}-fa^{m-1}=0,\\
h&a^{m-2}-ga^{m-1}=e^{-im\theta}\delta f,\\
-e^{-im\theta}\delta f&a+ha^{m-1}=e^{-im\theta}\delta g.
\end{aligned}
$$
Write the system in the matrix form
\begin{equation}
M'_{m-1}
\left(
\begin{array}{c}
a\\
a^2\\
\vdots\\
a^{m-1}
\end{array}
\right)
=
\left(
\begin{array}{c}
-h\\
0\\
\vdots\\
0\\
e^{-im\theta}\delta f\\
e^{-im\theta}\delta g
\end{array}
\right)
                                               \label{4.23}
\end{equation}
with the $m\times(m-1)$-matrix that has the block structure
$$
M'_{m-1}=\left(\begin{array}{c}M_{m-1}\\-e^{-im\theta}\delta f\ 0\ \cdots\cdots 0\ h\end{array}\right),
$$
where $M_{m-1}$ coincides, for $n=m-1$, with the three-diagonal $n\times n$-matrix
$$
M_n=\left(
\begin{array}{ccccccccccc}
-g&-f&0&0&0&\dots&0&0&0&0&0\\
h&-g&-f&0&0&\dots&0&0&0&0&0\\
0&h&-g&-f&0&\dots&0&0&0&0&0\\
\cdot&\cdot&\cdot&\cdot&\cdot&\cdot&\cdot&\cdot&\cdot&\cdot&\cdot\\
\cdot&\cdot&\cdot&\cdot&\cdot&\cdot&\cdot&\cdot&\cdot&\cdot&\cdot\\
\cdot&\cdot&\cdot&\cdot&\cdot&\cdot&\cdot&\cdot&\cdot&\cdot&\cdot\\
0&0&0&0&0&\dots&0&h&-g&-f&0\\
0&0&0&0&0&\dots&0&0&h&-g&-f\\
0&0&0&0&0&\dots&0&0&0&h&-g
\end{array}\right).
$$
The matrix $M_n$ is well defined for $n\ge2$. Actually, we have already used $M_2$: system \eqref{4.21} is a partial case of \eqref{4.23}.

Let $w_n$ be the determinant of $M_n$. Developing the determinant with respect to the first row, we obtain the recurrent formula
\begin{equation}
w_n=-g w_{n-1}+hfw_{n-2}\quad\mbox{for}\quad n\ge 4.
                                          \label{4.24}
\end{equation}
Additionally, for any $n\ge 2$,
$$
w_n\textrm{ has a holomorphic extension to }\D
$$
since $f,g$, and $h$ do have such extensions.
Observe also that
\begin{equation}
w_2=g^2+fh,\quad w_3=-g(g^2+2fh).
                                            \label{4.26}
\end{equation}

First, assume that $w_{m-1}$ is not identically zero. Considering first $m-1$ equations of system \eqref{4.23} and inverting the matrix $M_{m-1}$, we obtain that $a (e^{im\theta}w_{m-1})$ has a holomorphic extension to $\D$.
Applying Lemma \ref{L4.1}, we obtain \eqref{0.5}.

Next, assume that both functions $w_{m-1}$ and $w_{m-2}$ are identically zero.
With the help of \eqref{4.24}, this implies $w_j=0$ for every $j\ge 2$. In particular, $w_2=w_3=0$ that gives, by \eqref{4.26}, $g^3=0$. This contradicts our previous assumption on $g$.
Therefore we assume for the rest of the proof that $w_{m-1}\equiv0$ but $w_{m-2}$ is not identically zero.

Let $M_{m-1}^{(1)}$ (respectively, $M_{m-1}^{(2)}$) be the $(m-1)\times(m-1)$-matrix obtained from $M'_{m-1}$ by deleting the $(m-1)$-th row (respectively, by deleting the first row).
Denote the determinant of  $M_{m-1}^{(i)}$ by $w_{m-1}^{(i)} \ (i=1,2)$. A straightforward computation shows that
\begin{equation}
w_{m-1}^{(1)}= h w_{m-2}-e^{-im\theta}\delta f^{m-1},\quad
w_{m-1}^{(2)}=(-1)^{m-1}e^{-im\theta}\delta f w_{m-2}+ h^{m-1}.
                                \label{4.27}
\end{equation}
If either $w_{m-1}^{(1)}$ or $w_{m-1}^{(2)}$ is not identically zero, we can conclude as we did before: either  $a(e^{im\theta}w_{m-1}^{(1)})$ or
$a (e^{im\theta}w_{m-1}^{(2)})$ has a holomorphic extension to $\D$, and then obtain \eqref{0.5} by Lemma \ref{L4.1}.

Finally, assume that $w_{m-1}^{(1)}\equiv w_{m-1}^{(2)}\equiv0$. Then \eqref{4.27} implies
$$
\delta^2 f^m=(-1)^m (e^{2i\theta}h)^m.
$$
Since $\delta=a^me^{im\theta}$, this can be written as
$$
(a^2)^mf^m=(-h)^m.
$$
Applying Lemma \ref{L4.1} with $(b,\varphi,\psi,r,s)=(a^2,f,-h,m,m)$, we obtain
$$
a^2(\theta)={\hat c}_0+2\Re({\hat c}_1e^{i\theta}).
$$
In particular, $e^{2i\theta}a^2$ admits a holomorphic extension to $\D$. Now, we can proceed in the same way as in the above-considered case of $m=2$. Namely, multiply equation \eqref{4.15} by $e^{2i\theta}$
$$
ae^{2i\theta}g=e^{2i\theta}h-e^{2i\theta}a^2f.
$$
The right-hand side admits a holomorphic extension to $\D$. Applying Lemma \ref{L4.1} with
$(b,\varphi,\psi,r,s)=(a,e^{2i\theta}g,e^{2i\theta}h-e^{2i\theta}a^2f,1,2)$, we obtain \eqref{0.5}.
\end{proof}

\begin{proof}[Proof of Lemma \ref{L4.1}]
Taking the $s$th power of both sides of \eqref{4.18}, we obtain $b^{rs}\varphi^{rs}=\psi^{rs}$ and
$b^{rs}e^{irs\theta}$ admits a holomorphic extension to $\D$. The lemma is thus reduced to the case of $r=s$.

Assume hypotheses of Lemma \ref{L4.1} to be satisfied with $r=s$. Let $\Delta$ be the holomorphic extension of $b^re^{ir\theta}$ to $\D$.
We assume that $\Delta$ is not identically zero, otherwise $b\equiv0$ and there is nothing to prove.
 Since $b$ is a real function, we have
\begin{equation}
e^{ir\theta}\overline{\Delta(e^{i\theta})}=e^{-ir\theta}\Delta (e^{i\theta}).
                                                          \label{4.28}
\end{equation}
With the help of the Cauchy theorem, we derive from the last equation
$$
\int_0^{2\pi}\Delta(e^{i\theta})e^{-i(\ell+2r)\theta}d\theta=
\int_0^{2\pi}\overline{\Delta(e^{i\theta})}e^{-i\ell\theta}d\theta
=\overline{\int_0^{2\pi}\Delta(e^{i\theta})e^{i\ell\theta}d\theta}=0
$$
for $\ell>0$. Therefore $\Delta(z)$ is a polynomial of degree at most $2r$. In addition, as is seen from \eqref{4.28},
\begin{equation}
\overline{\Delta(1/{\bar z})}=z^{-2r}\Delta (z)\quad\mbox{for}\quad z\in \C\b\{0\}.
                                                          \label{4.29}
\end{equation}
If
\begin{equation}
\Delta (z)=\alpha_{2r}z^{2r}+\alpha_{2r-1}z^{2r-1}+\dots+\alpha_{1}z+\alpha_0,
                                                          \label{4.30}
\end{equation}
then equation \eqref{4.29} is equivalent to
\begin{equation}
\alpha_{i}=\overline{\alpha_{2r-i}}\quad(0\le i\le 2r).
                                                          \label{4.31}
\end{equation}

Let $\Phi$ and $\Psi$ be analytic extensions to $\D$ of functions $\varphi$ and $\psi$ respectively. By the uniqueness theorem for holomorphic functions,
\eqref{4.18} implies
$$
z^r\Psi^r(z)=\Delta(z)\Phi^r(z)\quad\textrm{for}\quad z\in \D.
$$
This implies: if $z_0\in\,\stackrel\circ{\D}$ is a zero of order $k$ for the function $\Phi(z)$, then $z_0$ is zero of order $\ge k$ for $z\Psi(z)$.
Therefore $H(z)=z\Psi(z)/\Phi(z)$ is a well defined holomorphic function on $\stackrel\circ{\D}$ and
\begin{equation}
H^r(z)=\Delta(z)\quad\textrm{for}\quad z \in \,\stackrel\circ{\D}.
                                                         \label{4.33}
\end{equation}

First, assume $\Delta$ to have a zero $0\neq z_0\in\C\setminus\S$. By \eqref{4.29}, $1/{\bar z}_0$ is also a zero for $\Delta$ of the same order.
We can assume $|z_0|<1$, otherwise change roles of $z_0$ and $1/{\bar z}_0$.
By \eqref{4.33}, $z_0$ is a zero of $H$. Therefore $z_0$ and $1/{\bar z}_0$ are zeros of order $\ge r$ for $\Delta=H^r$. Since $\Delta$ is a polynomial of degree at most $2r$, we see that
\begin{equation}
\Delta(z)=c^r(z-z_0)^r(z-\bar z_0^{-1})^r
                                                  \label{4.34}
\end{equation}
for some constant $c\neq0$. Thus, $z_0$ is a simple zero of $H$. Now, \eqref{4.33} and \eqref{4.34} imply that
$$
\left|{H(z)\over c(z-z_0)(z-\bar z_0^{-1})}\right|=1\quad\mbox{for}\quad z\in\,\stackrel\circ{\D}.
$$
By the maximum principle, there exists a constant $c_1\not=0$ such that
$$
H(z)=c_1(z-z_0)(z-\bar z_0^{-1}).
$$
Now, we have
$$
c_1^r(z-z_0)^r(z-\bar z_0^{-1})^r=z^rb^r(z)\quad\mbox{for}\quad z\in \S.
$$
Since the left-hand side does not vanish on $\S$, this implies the existence of a constant $c_2\in \C$ such that
$$
b(z)=c_2z^{-1}(z-z_0)(z-\bar z_0^{-1})\quad\mbox{for}\quad z\in \S.
$$
For a real function $b$, the last equation implies \eqref{4.19}.

Next, assume $0$ to be a zero of $\Delta$. Again $0$ is a zero of $H$ and then $0$ is a zero of order $\ge r$ for $\Delta$, i.e.,
$\alpha_0=\dots=\alpha_{r-1}=0$ on \eqref{4.30}. With the help of \eqref{4.31}, this implies the existence of a constant $c\neq0$ such that  $\Delta=c^rz^r$. Hence $b$ is a constant function and \eqref{4.19} holds.

Finally, assume that all zeroes of the polynomial $\Delta$ belong to $\S$ and let $z_0$ be one of them. Since $\Delta(z)=z^rb^r(z)$ for $z\in \S$, we see that $b(z_0)=0$. This implies that $z_0$ is a zero of order $\ge r$ for $\Delta$. Therefore $\Delta$ has either two distinct zeros of order $r$ or a single zero of order $2r$. Therefore there exist a constant $c$ and $z_1,z_2\in \S$ such that
$$
\Delta(z)=c(z-z_1)^r(z-z_2)^r=b^r(z)z^r\quad\mbox{for}\quad z\in \S.
$$
Since $b$ is an infinitely smooth function, this implies the existence of a constant $c_1\in \C$ such that
$$
b(z)=c_1z^{-1}(z-z_1)(z-z_2).
$$
Again, the latter equation implies \eqref{4.19} for a real function $b$.
\end{proof}

\section{The compactness theorem}

The following statement belongs to Edward \cite[Proposition 1]{E2}. The proof is presented in \cite{E0}. Since the latter paper is not easily accessible, we present a proof in Appendix.

\begin{lemma} \label{L5.1}
Given a positive function $a\in C^\infty(\S)$, there exists a function $b\in C^\infty(\S)$ which is conformally equivalent to $a$ via a conformal transformation of the disk $\D$ and such that ${\hat b}_1=0$.
\end{lemma}

We will need one more Edward's result \cite[Propositions 3 and 4]{E2}:

\begin{lemma} \label{L5.2}
Let $a_n\in C^\infty(\S)\ (n=1,2,\dots)$ be a sequence of positive functions such that the Steklov spectrum $\mbox{\rm Sp}(a_n)$ is independent of $n$.
Assume additionally that $(\widehat{a_n})_1=0$ for every $n$. Then
\begin{equation}
|(\widehat{a_n})_0|\le C_0
                                                  \label{5.1}
\end{equation}
and
\begin{equation}
a_n(\theta)\ge c\quad(\theta\in\R)
                                                  \label{5.2}
\end{equation}
with positive constants $c$ and $C_0$ independent of $n$.
\end{lemma}

Let us remind that, for every real $s$, the Hilbert space $H^s({\mathbb S})$ is the completion of $C^\infty({\mathbb S})$ with respect to the norm
\begin{equation}
\|u\|_{H^s({\mathbb S})}^2=\sum\limits_{n\in\Z}(1+|n|^{2s})|{\hat u}_n|^2.
                                                  \label{5.2a}
\end{equation}

\begin{lemma} \label{L5.3}
Let $a_n\in C^\infty(\S)\ (n=1,2,\dots)$ be a sequence of positive functions such that the Steklov spectrum $\mbox{\rm Sp}(a_n)$ is independent of $n$.
Assume additionally that $(\widehat{a_n})_1=0$ for every $n$. Then the sequence is bounded in $H^s({\mathbb S})$ for every $s\in\R$, i.e.,
\begin{equation}
\|a_n\|_{H^s({\mathbb S})}\le C_s
                                                  \label{5.3}
\end{equation}
with a constant $C_s$ independent of $n$.
\end{lemma}

\begin{proof}
Zeta-invariants $Z_m(a_n)$ are independent of $n$.
By Edward's formula \eqref{0.3},
$$
Z_1(a_n)=\frac{2}{3}\sum\limits_{k\ge2}(k^3-k)|(\widehat{a_n})_k|^2\ge
\frac{4}{9}\sum\limits_{k\ge2}(1+k^3)|(\widehat{a_n})_k|^2
=\frac{2}{9}\Big(\|a_n\|_{H^{3/2}({\mathbb S})}^2-|(\widehat{a_n})_0|^2\Big).
$$
We have used that $(\widehat{a_n})_1=0$. With the help of \eqref{5.1}, this implies
$$
\|a_n\|_{H^{3/2}({\mathbb S})}^2\le \frac{9}{2}Z_1(a_n)+|(\widehat{a_n})_0|^2\le\frac{9}{2}Z_1(a_n)+C_0^2.
$$
Thus, the sequence $a_n$ is uniformly bounded in $H^{3/2}({\mathbb S})$.
Since $H^{3/2}({\mathbb S})\subset C(\S)$, we have in particular $a_n(\theta)\le C\ (\theta\in\R)$ with a constant $C$ independent of $n$.
Combine this with \eqref{5.2}
\begin{equation}
0<c\le a_n(\theta)\le C\quad(\theta\in\R).
                                                  \label{5.5}
\end{equation}

Now, we prove the statement: for every integer $m\ge1$, the sequence $a_n^m\ (n=1,2,\dots)$ is uniformly bounded in $H^{m+1/2}({\mathbb S})$. Indeed, \eqref{5.5} gives
\begin{equation}
c^m\le a^m_n(\theta)\le C^m\quad(\theta\in\R).
                                                  \label{5.7}
\end{equation}
This implies
\begin{equation}
|(\widehat{a^m_n})_k|=\left|\frac{1}{2\pi}\int\limits_0^{2\pi}e^{-ik\theta}a^m_n(\theta)\,d\theta\right|\leq C^m.
                                                  \label{5.8}
\end{equation}
By Theorem \ref{Th0.1},
$$
\sum\limits_{k=m+1}^\infty k^{2m+1}|(\widehat{a^m_n})_k|^2\le c^{-1}_{m}Z_{m}(a_n).
$$
Therefore
\begin{equation}
\sum\limits_{k=m+1}^\infty (1+k^{2m+1})|(\widehat{a^m_n})_k|^2\le 2c^{-1}_{m}Z_{m}(a_n).
                                                  \label{5.9}
\end{equation}
Estimates \eqref{5.8} and  \eqref{5.9} imply
$$
\begin{aligned}
\|a^m_n\|_{H^{m+1/2}({\mathbb S})}^2&=\sum\limits_{|k|\le m}(1+|k|^{2m+1})|(\widehat{a^m_n})_k|^2+2\sum\limits_{k\ge m+1}(1+|k|^{2m+1})|(\widehat{a^m_n})_k|^2\\
&\le (2m+1)\big(1+m^{2m+1}\big)C^{2m}+4c^{-1}_mZ_m(a_n).
\end{aligned}
$$
Thus,
$$
\|a^m_n\|_{H^{m+1/2}({\mathbb S})}\le C_m,
$$
where $C_m^2=(2m+1)\big(1+m^{2m+1}\big)C^{2m}+4c^{-1}_mZ_m(a_n)$ is independent of $n$.
Since $\|\cdot\|_{H^{m}({\mathbb S})}\le\|\cdot\|_{H^{m+1/2}({\mathbb S})}$, we have
\begin{equation}
\|a^m_n\|_{H^{m}({\mathbb S})}\le C_m\quad(m=1,2,\dots).
                                                  \label{5.10}
\end{equation}

For an integer $s=m$, norm \eqref{5.2a} can be equivalently written as
$$
\|u\|^2_{H^{m}({\mathbb S})}=\sum\limits_{k=0}^m\|D^k u\|^2_{L^2({\mathbb S})}.
$$
On using this definition and estimates \eqref{5.5}, one easily proves the equivalence of the norms $\|a_n\|_{H^{m}({\mathbb S})}$ and $\|\log(a_n)\|_{H^{m}({\mathbb S})}$, i.e., the validity of estimates
\begin{equation}
\frac{1}{C'_m}\|a_n\|_{H^{m}({\mathbb S})}\le \|\log(a_n)\|_{H^{m}({\mathbb S})}\le C'_m\|a_n\|_{H^{m}({\mathbb S})}
                                                  \label{5.11}
\end{equation}
with a constant $C'_m$ independent of $n$. In view of \eqref{5.7}, the same is true for $a^m_n$
\begin{equation}
\frac{1}{C''_m}\|a^m_n\|_{H^{m}({\mathbb S})}\le \|\log(a^m_n)\|_{H^{m}({\mathbb S})}\le C''_m\|a^m_n\|_{H^{m}({\mathbb S})}.
                                                  \label{5.12}
\end{equation}

Finally, we prove the uniform boundedness of the sequence $a_n$ in $H^{m}({\mathbb S})$. Since $\log(a_n)=\frac{1}{m}\log(a^m_n)$, estimates \eqref{5.10}--\eqref{5.12} give
$$
\begin{aligned}
\|a_n\|_{H^{m}({\mathbb S})}\le C'_m\|\log(a_n)\|_{H^{m}({\mathbb S})}&=\frac{C'_m}{m}\|\log(a^m_n)\|_{H^{m}({\mathbb S})}\\
&\le\frac{C'_mC''_m}{m}\|a^m_n\|_{H^{m}({\mathbb S})}\le\frac{C_mC'_mC''_m}{m}.
\end{aligned}
$$
Thus, the sequence $a_n\ (n=1,2,\dots)$ is uniformly bounded in $H^{m}({\mathbb S})$ for every integer $m$ and hence in $H^{s}({\mathbb S})$ for every real $s$.
\end{proof}

\begin{proof}[Proof of Theorem \ref{Th0.3}.]
Let a sequence $a_n\in C^\infty(\S)\ (n=1,2,\dots)$ of positive functions satisfy hypotheses of Theorem \ref{Th0.3}. With the help of Lemma \ref{L5.1}, we can assume without loss of generality that $(\widehat{a_n})_1=0$ for every $n$. Therefore estimates \eqref{5.1} and \eqref{5.2} are valid as well as estimate \eqref{5.3} is valid for every $s\in\R$. Since the embedding $H^1(\S)\subset C(\S)$ is compact, we can choose a subsequence converging in $C(\S)$. The limit function
$a\in C(\S)$ satisfy
$$
a(\theta)\ge c\quad(\theta\in\R)
$$
as follows from \eqref{5.2}. Since the embedding $H^1(\S)\subset H^2(\S)$ is compact, we can choose a sub-subsequence converging to $a$ in $H^1(\S)$, and so on.
On using the classical trick of choosing the diagonal sequence, we obtain a subsequence $a_{n_k}\ (k=1,2,\dots)$ which converges to $a$ in $H^s(\S)$ for every
$s\in\R$. In other words, $a_{n_k}$ converges to $a$ in $C^\infty(\S)$. In particular, $a\in C^\infty(\S)$.
\end{proof}

\bigskip

Finally, we give an interpretation of Theorem \ref{Th0.3} in terms of Steklov isospectral families of planar domains. Let $\Omega_n\ (n=1,2,\dots)$ be a sequence of smooth bounded simply connected (probably multisheet) planar domains. We say that the sequence converges in the $C^\infty_{hol}(\D)$-topology if, for every $n$, there exists a biholomorphism $\Phi_n:\D\rightarrow\Omega_n$ such that the sequence $\Phi_n$ converges to some $\Phi:\D\rightarrow{\R}^2$ in the $C^\infty(\D)$-topology; the set $\Phi(\D)$ is the limit of the sequence.

\begin{theorem} \label{Th5.1}
Let $\Omega_n\ (n=1,2,\dots)$ be a sequence of smooth bounded simply connected (probably multisheet) planar domains. Assume that the Steklov spectrum $\mbox{\rm Sp}(\Omega_n)$ is independent of $n$. Then there exists a subsequence $\Omega_{n_k}\ (k=1,2,\dots)$ such that, for appropriately chosen isometries $I_k:{\R}^2\rightarrow{\R}^2$, the sequence $I_k(\Omega_{n_k})$ converges in the $C^\infty_{hol}(\D)$-topology to some smooth bounded simply connected multisheet planar domain.
\end{theorem}

\begin{proof}
By the Riemann theorem, there exists a biholomorphism $\Phi_n:\D\rightarrow\Omega_n$ for every $n$. Define positive functions $a_n\in C^\infty(\S)$ by $a_n(z)=|\Phi'_n(z)|^{-1}$ for $z\in\S$. The Steklov spectrum $\mbox{\rm Sp}(a_n)=\mbox{\rm Sp}(\Omega_n)$ is independent of $n$.
By Lemma \ref{L5.1}, for every $n$, there exists a  conformal transformation $\Psi_n:\D\rightarrow\D$ such that the function
$b_n(z)=a_n\big(\Psi_n(z)\big)|\Psi'_n(z)|^{-1}$ satisfies $(\widehat{b_n})_1=0$.
Then $b_n(z)=|{\tilde\Phi}'_n(z)|^{-1}\ (z\in\S)$ for the diffeomorphism $\tilde\Phi_n=\Phi_n\circ\Psi_n:\D\rightarrow\Omega_n$.

All $\Phi_n$ are biholomorphisms. In order to simplify notations, we assume without loss of generality that all $\Psi_n$ are identities, i.e., that $\tilde\Phi_n=\Phi_n$.

Thus, for every $n$, we have a biholomorphism $\Phi_n:\D\rightarrow\Omega_n$ such that the function $a_n(z)=|\Phi'_n(z)|^{-1}\ (z\in\S)$ satisfies $(\widehat{a_n})_1=0$ and $\mbox{\rm Sp}(a_n)$ is independent of $n$. By Lemma \ref{L5.2} (see also \eqref{5.5}),
\begin{equation}
0<c\le a_n(\theta)\le C\quad(\theta\in\R)
                                                  \label{5.13}
\end{equation}
with constants $c$ and $C$ independent of $n$. By Lemma \ref{L5.3}, the sequence $a_n\ (n=1,2,\dots)$ contains a subsequence converging in $C^\infty(\S)$. Without loss of generality, we assume that the sequence $a_n$ itself to converge in $C^\infty(\S)$ to some function $a\in C^\infty(\S)$. The limit function also satisfies
\begin{equation}
c\le a(\theta)\le C\quad(\theta\in\R)
                                                  \label{5.14}
\end{equation}
as follows from \eqref{5.13}.

For every $n$, the function $\Phi_n$ is holomorphic on $\stackrel\circ{\D}$, continuous with all its derivatives on $\D$, and satisfies the boundary condition
$$
|\Phi'_n(z)|=a_n^{-1}(z)\quad(z\in\S).
$$
The derivative $\Phi'_n(z)$ does not vanish in $\D$ because $\Phi_n$ is a biholomorphism. Therefore $\log\Phi'_n(z)$ is a well defined holomorphic function on $\stackrel\circ{\D}$ whose real part $u_n(z)=\Re(\log\Phi'_n(z))$ solves the Dirichlet problem
$$
\Delta u_n=0\quad\mbox{in}\quad\stackrel\circ{\D},\quad u_n|_{\S}=\log(a_n^{-1}).
$$
The sequence $a_n^{-1}$ converges to $a^{-1}$ in $C^\infty(\S)$. Applying standard Sobolev stability estimates for an elliptic boundary value problem, we see that the sequence $u_n$ converges in the $C^\infty(\D)$-topology to a function $u\in C^\infty(\D)$ which solves the Dirichlet problem
$$
\Delta u=0\quad\mbox{in}\quad\stackrel\circ{\D},\quad u|_{\S}=\log(a^{-1}).
$$
The function $u$ satisfies
\begin{equation}
-\log(C)\le u(z)\le -\log(c)\quad(z\in\D)
                                                  \label{5.16}
\end{equation}
as follows from \eqref{5.14} with the help of the maximum principle for harmonic functions.

Now, we address the equation
\begin{equation}
\Re(\log\Phi'_n(z))=u_n(z).
                                                  \label{5.17}
\end{equation}
A holomorphic function $f(z)$ is defined by its real part uniquely up to a pure imaginary constant. Moreover, the dependence of $f$ on $\Re(f)$ is continuous in corresponding Sobolev norms. Therefore \eqref{5.17} implies with the help of \eqref{5.16} the existence of a function $\Phi\in C^\infty(\D)$ which is holomorphic
in $\stackrel\circ{\D}$ and such that the sequence
\begin{equation}
\alpha_n\Phi_n+\beta_n\rightarrow\Phi\quad\mbox{in}\quad C^\infty(\D)
                                                  \label{5.18}
\end{equation}
with appropriately chosen constants $\alpha_n\in\C,\ |\alpha_n|=1$ and $\beta_n\in\C$. The derivative $\Phi'(z)$ does not vanish in $\D$ because $|\Phi'(z)|=\exp\big(u(z)\big)$. Therefore $\Omega=\Phi(\D)$ is a smooth simply connected (probably multisheet) planar domain. Define the isometry $I_n:{\R}^2\rightarrow{\R}^2$ by $I_n(z)=\alpha_nz+\beta_n$. Then \eqref{5.18} can be rewritten as
$$
I_n\circ\Phi_n\rightarrow\Phi\quad\mbox{in}\quad C^\infty(\D).
$$
This means that the sequence of domains $I_n(\Omega_n)=(I_n\circ\Phi_n)(\D)$ converges to $\Omega=\Phi(\D)$ in the $C^\infty_{hol}(\D)$-topology.
\end{proof}

\appendix
\section{Proofs of Lemmas \ref{L1.1} and \ref{L5.1}}

We will need the following
\begin{lemma} \label{LA.1}
Let $A$ and $B$ be pseudodifferential operators on $\S$. If $B$ is a smoothing operator, then $\T[AB]=\T[BA]$.
\end{lemma}

\begin{proof}
Denote by $P_0$ the bounded operator on $L^2(\S)$ defined as the orthogonal projector onto the one-dimensional subspace of constant functions: $P_0(f)=\hat f_0$ for $f\in L^2(\S)$. Then $\Lambda+P_0$ is a first order elliptic positive pseudodifferential operator. For every $\mu\in{\mathbb R}$, the power $(\Lambda+P_0)^\mu$ is a well defined pseudodifferential operator of order $\mu$.

Let $\mu$ be the order of the pseudodifferential operator $A$. Then $A(\Lambda+P_0)^{-\mu}$ is a bounded operator on $L^2(\S)$ and $(\Lambda+P_0)^\mu B$ is a smoothing operator. Then by the classical property of the trace \cite[Theorem 3.1]{Si}
$$
\T[AB]=\T[A(\Lambda+P_0)^{-\mu}(\Lambda+P_0)^\mu B]=\T[(\Lambda+P_0)^\mu B A(\Lambda+P_0)^{-\mu}].
$$
We compute the right-hand side on using the trigonometric basis:
$(\Lambda+P_0)e^{in\theta}=(|n|+1)e^{in\theta}$ and
$$
\l (\Lambda+P_0)^\mu B A(\Lambda+P_0)^{-\mu}e^{in\theta},e^{in\theta}\r=\l B A e^{in\theta},e^{in\theta}\r.
$$
Two last formulas give $\T[AB]=\T[(\Lambda+P_0)^\mu B A(\Lambda+P_0)^{-\mu}]=\T[BA]$.
\end{proof}

\begin{proof}[Proof of Lemma \ref{L1.1}.]
It suffices to prove that
\begin{equation}
\T\big[(L^{j_1}H)(L^{j_2}H)\dots(L^{j_{2s}}H)-(LH)^{2m}\big]=\T\big[(L^{j_2}H)\dots(L^{j_{2s}}H)(L^{j_1}H)-(LH)^{2m}\big]
                                          \label{A.2}
\end{equation}
for a sequence of non-negative integers $(j_1,\dots,j_{2s})$ satisfying $j_1+\dots+j_{2s}=2m$. We prove this by induction in $s$.

First let $s=1$ and $0\le j\le 2m$. Then
\begin{equation}
\T\big(L^jHL^{2m-j}H-(LH)^{2m}\big)=\T\big([H,L^{2m-j}]HL^j\big)+\T\big(L^{2m}-(LH)^{2m}\big).
                                          \label{A.3}
\end{equation}
Applying Lemma \ref{LA.1} with $A=L^j$, $B=[H,L^{2m-j}]H$ and using $H^2=I$, we have
$$
\begin{aligned}
\T\big(L^j[H,L^{2m-j}]H\big)&=\T\big([H,L^{2m-j}]HL^j\big)\\
&=\T\big(HL^{2m-j}HL^j-L^{2m}\big)=\T\big(L^{2m-j}HL^jH-L^{2m}\big).
\end{aligned}
$$
Substituting this value into \eqref{A.3}, we obtain \eqref{A.2} for $s=1$.

Assume \eqref{A.2} to be valid for some $s\geq1$. Let $j_1+\dots+j_{2s+2}=2m$. Then
\begin{equation}
\begin{aligned}
\T\Big((L^{j_1}H)\dots(L^{j_{2s+2}}H)&-(LH)^{2m}\Big)
=\T\Big((L^{j_1}H)\dots(L^{j_{2s}}H)[L^{j_{2s+1}},H]L^{j_{2s+2}}H\Big)\\
&+\T\Big((L^{j_1}H)\dots(L^{j_{2s-1}}H)L^{j_{2s}+j_{2s+1}+j_{2s+2}}H-(LH)^{2m}\Big).
\end{aligned}
                                          \label{A.4}
\end{equation}
Applying Lemma \ref{LA.1} with $A=L^{j_1}H$ and $B=(L^{j_2}H)\dots(L^{j_{2s}}H)[L^{j_{2s+1}},H]L^{j_{2s+2}}H$, we obtain
\begin{equation}
\begin{aligned}
&\T\Big((L^{j_1}H)\dots(L^{j_{2s}}H)[L^{j_{2s+1}},H]L^{j_{2s+2}}H\Big)\\
&=\T\Big((L^{j_2}H)\dots(L^{j_{2s}}H)[L^{j_{2s+1}},H]L^{j_{2s+2}}H(L^{j_1}H)\Big)\\
&=\T\Big((L^{j_2}H)\dots(L^{j_{2s+2}}H)(L^{j_1}H)
-(L^{j_2}H)\dots(L^{j_{2s-1}}H)L^{j_{2s}+j_{2s+1}+j_{2s+2}}H(L^{j_1}H)\Big).
\end{aligned}
                                          \label{A.5}
\end{equation}
By the induction hypothesis,
\begin{equation}
\begin{aligned}
\T\Big((L^{j_1}H)\dots&(L^{j_{2s-1}}H)L^{j_{2s}+j_{2s+1}+j_{2s+2}}H-(LH)^{2m}\Big)\\
&=\T\Big((L^{j_2}H)\dots(L^{j_{2s-1}}H)L^{j_{2s}+j_{2s+1}+j_{2s+2}}H(L^{j_1}H)-(LH)^{2m}\Big).
\end{aligned}
                                          \label{A.6}
\end{equation}
Substituting \eqref{A.5} and \eqref{A.6} into \eqref{A.4}, we obtain
$$
\begin{aligned}
\T\big[(L^{j_1}H)\dots(L^{j_{2s+2}}H)-(LH)^{2m}\big]
=\T\big[(L^{j_2}H)\dots(L^{j_{2s+2}}H)(L^{j_1}H)-(LH)^{2m}\big].
\end{aligned}
$$
\end{proof}

\begin{proof}[Proof of Lemma \ref{L5.1} {\rm (by \cite{Sz} and \cite{We})}.]
Let $a\in C^\infty(\S)$ be a positive function.
For $r\in[0,1)$ and $\alpha\in\R$, let $b_{r,\alpha}\in C^\infty(\S)$ be the function conformally equivalent to $a$ via the conformal transformation
$\Phi_{r,\alpha}:\D\rightarrow\D$ defined by
$$
\Phi_{r,\alpha}(z)=e^{i\alpha}\frac{z-r}{1-rz},
$$
i.e.,
\begin{equation}
b_{r,\alpha}(z)=a(\Phi_{r,\alpha}(z))|\Phi'_{r,\alpha}(z)|^{-1}=a\Big(e^{i\alpha}\frac{z-r}{1-rz}\Big)\frac{|1-rz|^2}{1-r^2}\quad (z\in\S).
                                          \label{A.7}
\end{equation}
Define the function $H\in C^\infty([0,1)\times\R)$ by
\begin{equation}
H(r,\alpha)= (1-r^2)e^{-i\alpha}(\widehat{b_{r,\alpha}})_{1}={1\over 2\pi}\int_0^{2\pi}e^{-i(\alpha+\theta)}a\Big(e^{i\alpha}\frac{e^{i\theta}-r}{1-re^{i\theta}}\Big)|1-re^{i\theta}|^2\,d\theta.
                                          \label{A.8}
\end{equation}
Observe that $H(0,\alpha)={\hat a}_1$ for every $\alpha$. Therefore $H$ can be considered as the continuous function $H:{\stackrel\circ{\D}}\rightarrow\C$ such that $H(re^{i\alpha})=H(r,\alpha)$.

The integrand in \eqref{A.8} is bounded uniformly in $(r,\alpha,\theta)\in[0,1)\times\R\times\R$ and
$$
a\Big(e^{i\alpha}\frac{e^{i\theta}-r}{1-re^{i\theta}}\Big)\rightarrow a(-e^{i\alpha})\quad\mbox{as}\quad r\rightarrow1-0.
$$
Therefore
\begin{equation}
H(r,\alpha)= a(-e^{i\alpha})e^{-i\alpha}{1\over 2\pi}\int_0^{2\pi}e^{-i\theta}|1-e^{i\theta}|^2\,d\theta+E(r,\alpha),
                                          \label{A.9}
\end{equation}
where the error term $E(r,\alpha)\rightarrow0$ uniformly in $\alpha$ as $r\rightarrow1-0$.
Substituting the value
$$
{1\over 2\pi}\int_0^{2\pi}e^{-i\theta}|1-e^{i\theta}|^2\,d\theta=
{1\over 2\pi}\int_0^{2\pi}e^{-i\theta}(2-e^{i\theta}-e^{-i\theta})\,d\theta=-1
$$
into \eqref{A.9}, we obtain
\begin{equation}
H(r,\alpha)= -a(-e^{i\alpha})e^{-i\alpha}+E(r,\alpha).
                                          \label{A.10}
\end{equation}

Since $a$ is a positive function, \eqref{A.10} implies that the closed curve $\alpha\mapsto H(r,\alpha)\ (0\le\alpha\le2\pi)$ goes around 0 if $r\in(0,1)$ is sufficiently close to 1. Therefore there exists $(r_0,\alpha_0)\in[0,1)\times\R$ such that $H(r_0,\alpha_0)=0$. Together with \eqref{A.8}, this gives
$(\widehat{b_{r_0,\alpha_0}})_{1}=0$.
\end{proof}

\end{document}